\documentclass{article}

\usepackage{lipsum}
\usepackage{amsfonts}
\usepackage{graphicx}
\usepackage{epstopdf}
\usepackage{algorithmic}

\usepackage{amsmath}
\usepackage{amsfonts}
\usepackage{graphicx, subfigure}
\usepackage[pdftex]{hyperref}
\usepackage[dutch,english]{babel}
\usepackage{amsmath,amssymb}
\usepackage{amsthm}
\usepackage[usenames,dvipsnames]{color}
\usepackage{tikz}
\usepackage{accents}
\usepackage{xcolor}
\usepackage{hyperref}
\usepackage{mathtools}
\usepackage{multicol}
\usepackage{comment}

\theoremstyle{plain}
\newtheorem{theorem}{Theorem}[section]

\newtheorem{lemma}[theorem]{Lemma}
\newtheorem{prop}[theorem]{Proposition}
\newtheorem{cor}[theorem]{Corollary}
\theoremstyle{definition}
\newtheorem{defn}[theorem]{Definition}

\newtheorem{hyp}{Hypothesis}
\newtheorem{step}{Summary of Step}
\theoremstyle{remark}
\newtheorem{remark}{Remark}

\title{
Higher-order phase reduction for delay-coupled oscillators beyond the phase-shift approximation
}

\author{Christian Bick, Bob W. Rink and Babette A. J. de Wolff}

\usepackage{amsmath}
\usepackage{amsfonts}
\usepackage{graphicx, subfigure}
\usepackage{mathtools}
\usepackage{multicol}
\usepackage{comment}

\DeclareMathOperator{\im}{im}

\newcommand{\ord}[2]{{#1}^{(#2)}}
\newcommand{\comp}[2]{{#1}_{#2}}
\newcommand{\torus}[1]{\mathbb{T}^{#1}}
\newcommand{\statedim}{n} 
\newcommand{\torusdim}{m} 
\newcommand{\cc}{\rho} 

\DeclarePairedDelimiter{\dotp}{\langle}{\rangle}
\renewcommand{\phi}{\varphi}
\renewcommand{\epsilon}{\varepsilon}

\newcommand{\Tg}{T}

\newcommand{\el}{\ell}
\newcommand{\El}{\mathcal{L}}

\begin{document}

\maketitle

\begin{abstract}
Network interactions between dynamical units are often subject to time delay.
We develop a phase reduction method for delay-coupled oscillator networks.
The method is based on rewriting the delay-differential equation as an ordinary differential equation coupled with a transport equation, expanding in the coupling strength, and solving the resulting equations order-by-order.
This approach yields an approximation of the finite-dimensional phase dynamics to arbitrary order.
While in the first-order approximation the time delay acts as a phase shift as expected, the higher-order phase reduction generally displays a less trivial dependence on the delay. 
  In particular, exploiting second-order phase reduction, we prove the existence of a region of bistability in the synchronization dynamics of two delay-coupled Stuart--Landau oscillators.
\end{abstract}

\section{Introduction}

\newcommand{\eps}{\varepsilon}

Interaction delays are often an essential part of coupled dynamical systems.
For example, transmission delays are common in networks where transmission along physical connections happens at finite speed.
To account for interaction delays, delay differential equations (DDEs) provide a suitable class of model equations to understand the collective network dynamics. 
A prototypical example is a network of~$m\in\mathbb{N}$ coupled dynamical units whose states $x_j\in\mathbb{R}^{\statedim_j}$ evolve according to
\begin{align}\label{eq:IntroNet}
    \dot x_j(t) &= F_j(x_j(t)) + \varepsilon \sum_{k=1}^\torusdim G_{j,k}\left(x_j(t), x_k(t-\tau_{j,k})\right),
\end{align}
where~$F_j$ determines the intrinsic dynamics of each unit, $\eps$~denotes the coupling strength, and~$G_{j,k}$ determines the influence node~$k$ has on node~$j$ with time delay~$\tau_{j,k} > 0$.

Synchronization is a common collective phenomenon if each unit is oscillatory, that is, the dynamics of each uncoupled node $x_j(t) = F_j(x_j(t))$ has a normally hyperbolic periodic orbit.
Time delays in the network coupling influence the synchronization dynamics~\cite{Pazo2016}, which is important for brain dynamics and function~\cite{Deco2009,Petkoski2016}. 
Conversely, time-delayed interactions can also be used to control the collective dynamics of entire networks~\cite{Kori2008,Kiss2018}.
However, the infinite dimensional nature of the dynamical system~\eqref{eq:IntroNet} makes it challenging to analyze the system directly, whether with analytical or numerical means, to understand how node and network properties shape the synchronization dynamics.

\newcommand{\R}{\mathbb{R}}
\newcommand{\tr}{\mathsf{T}}
\newcommand{\Tm}{\torus{\torusdim}}

Here we develop a systematic approach to \emph{phase reduction} to reduce the dimension of delay-coupled oscillator networks.
In contrast to traditional phase reduction approaches~\cite{Nakao2015,Pietras2019} and related reduction techniques for oscillatory dynamics with delay (see, e.g.,~\cite{Kotani2020,Nicks2024}), our approach builds on recent work~\cite{VonderGracht2023a} to compute an asymptotic expansion for an invariant torus in phase space and the dynamics thereon.
Note that for $\epsilon = 0$ the system~\eqref{eq:IntroNet} has a normally hyperbolic $m$-dimensional invariant torus---the product of the $m$~hyperbolic periodic orbits in each uncoupled oscillator---which persists for small coupling $\epsilon > 0$~\cite{Fenichel1972}.
Phase reduction yields the evolution equations for the phase variables~$\phi\in\Tm$ parameterizing the torus. This evolutions is determined by a (finite-dimensional) ordinary differential equation (ODE)
\begin{equation} \label{eq:phase_intro}
    \dot{\phi}(t) = \ord{f}{\eps}(\phi(t)).
\end{equation}
The ``physical'' oscillator states~$x=(x_1, \dotsc, x_\torusdim)\in\R^n$ and the phases~$\phi$ are related through maps
\begin{align} \label{eq:embeddings_intro}
    \ord{e}{\epsilon}&: \torus{\torusdim} \to \R^\statedim, \qquad 
 \ord{E}{\epsilon}: \torus{\torusdim} \times [-\tau, 0] \to \R^\statedim
\end{align}
which describe the phase and its history on the invariant torus, respectively.
We show that these maps send solutions of the phase equations~\eqref{eq:phase_intro} to solutions of the unreduced dynamics~\eqref{eq:IntroNet} (in a way to be made precise) exactly when they satisfy the \emph{conjugacy equations}
\begin{subequations}
\label{eq:conjeqn_intro}
\begin{align}
    \partial_{\phi}  \ord{e}{\epsilon}_j(\phi) \cdot \ord{f}{\eps}(\phi)  & =  F_j(\ord{e}{\epsilon}_j(\phi)) +\epsilon \sum_{k=1}^m G_{j,k}(\ord{e}{\epsilon}_j(\phi), \ord{E}{\epsilon}_k(\phi, -\tau_{j,k})),  
    \\ 
    \partial_{\phi} \ord{E}{\eps}_j(\phi, s)\cdot \ord{f}{\eps}(\phi)  & = \partial_{s} \ord{E}{\eps}_j(\phi, s),
\end{align}
\end{subequations}
where the subindex~$j$ denotes the $j$th~entry.

The main result of this manuscript is that we prove that the phase dynamics~$\ord{f}{\epsilon}$ together with $\ord{e}{\epsilon}, \ord{E}{\epsilon}$ can be computed in a constructive way.
Our result can be stated as the following theorem.
Although we prove it for a more general class of delay equations, it is formulated here for the specific equations~\eqref{eq:IntroNet} and its conjugacy equations~\eqref{eq:conjeqn_intro}.

\begin{theorem} 
\label{thm:intro}
The conjugacy equations ~\eqref{eq:conjeqn_intro} can be solved to arbitrary precision in~$\epsilon$,  yielding an approximation of the invariant torus (determined by $\ord{e}{\eps}$,~$\ord{E}{\eps}$) and the dynamics thereon (given by~$\ord{f}{\eps}$).
In particular, we can compute the phase dynamics~\eqref{eq:phase_intro} on the invariant torus up to arbitrary order in~$\epsilon$.  
\end{theorem}

Our approach to deriving  equations~\eqref{eq:conjeqn_intro} and proving Theorem~\ref{thm:intro} proceeds in four steps in Section~\ref{sec:method}. 
The first step is to show that solutions of the DDE~\eqref{eq:IntroNet} correspond to solutions of an ODE coupled to a transport equation; this essentially allows us to describe the dynamics of~\eqref{eq:IntroNet} by means of an ODE on a Banach space. 
Second, we conjugate the dynamics of the reduced ODE~\eqref{eq:phase_intro} to the dynamics of the ODE-transport system, and from there derive the conjugacy equations~\eqref{eq:conjeqn_intro}. 
In the third step, we formally expand     $\ord{e}{\eps}  = \ord{e}{0} + \varepsilon \ord{e}{1} + \epsilon^2 \ord{e}{2} + \dotsb, 
    \ord{E}{\eps} = \ord{E}{0} + \varepsilon \ord{E}{1} + \varepsilon^2 \ord{E}{2} + \dotsb, 
    \ord{f}{\eps} = \ord{f}{0} + \varepsilon \ord{f}{1} + \varepsilon^2 \ord{f}{2} + \dotsb$
and substitute these expansions into the conjugacy equations~\eqref{eq:conjeqn_intro}. Collecting terms of equal order in~$\epsilon$ we obtain a sequence of equations for $(\ord{f}{0}, \ord{e}{0}, \ord{E}{0})$, $(\ord{f}{1}, \ord{e}{1}, \ord{E}{1}), \ldots$. 
In the fourth and final step we then show that this sequence of equations can be solved via an iterative procedure. 
This then proves the statement in Theorem~\ref{thm:intro} 
by providing a constructive algorithm to compute~$\ord{f}{\eps}$ (and~$\ord{e}{\eps}$ and~$\ord{E}{\eps}$) up to arbitrary order in~$\epsilon$.
We employ the method to derive the first- and second-order phase reduction of delay-coupled Stuart--Landau oscillators (Section~\ref{sec:sl}), which allows to prove the existence of regions of bistability between phase synchrony and anti-phase synchrony  (Section~\ref{sec:bistab}).

Our mathematical approach to phase reduction for delay-coupled systems also gives perspective to existing work in the physics literature. First, we recover what is known as the \emph{phase shift approximation} $\phi_k(t - \tau_{j,k}) \approx \phi_k(t) - \omega_k \tau_{j,k}$. This approximation, used for example in~\cite{Kori2008,Smirnov2023}, 
replaces delays by phase-shift parameters. Our method produces exactly this approximation in the first-order phase equations, because it  corresponds to the solution to the zeroth-order transport equation. By contrast, delays enter the phase  equations in more complicated ways already at second order. Second, our phase reduction method computes an approximation of the invariant torus as a global object. This allows to deduce statements about the global dynamics, such as connecting orbits between states and basins of attractions. 
Third, the phase equations~\eqref{eq:phase_intro} are always finite-dimensional (as dynamics on a finite-dimensional torus). This results in simpler and qualitatively different phase reduced equations compared to phase equations with delay used for example in~\cite{Izhikevich1998}.   
Further discussion of our approach is given in Section~\ref{sec:discussion}.

\section{Phase reduction for delay-coupled networks}
\label{sec:method}

In this section, we present a systematic approach to phase reduction for general delay-coupled systems of the form
\begin{equation} \label{eq:DDE}
    \dot{x}(t) = F(x(t)) + \epsilon G(X(t, .)), \qquad \text{for}\ x\in \mathbb{R}^n \ \text{and} \ t \geq 0\, .
\end{equation}
Here the state variable is $x$ is an element of $\mathbb{R}^n$, the maximal time delay is $\tau > 0$ and the \emph{history segment} $X(t,\cdot) \in C([-\tau, 0], \mathbb{R}^n)$ is defined as 
\[
X(t,s) = x(t + s), \qquad s \in [-\tau, 0].
\]
Furthermore $G: C([-\tau, 0], \mathbb{R}^n) \to \mathbb{R}^n$ is a differentiable functional and we make the following assumptions on~$\epsilon$ and~$F$: 
\begin{hyp} \label{hyp:system} \hfill
   \begin{itemize}
       \item The function~$F$ is of the form 
       \[
       F(x)  = \left(\comp{F}{1}(x_1),  \ldots, \comp{F}{\torusdim}(x_\torusdim) \right)^\tr
       \] 
       with $
        x \in \mathbb{R}^\statedim : = \mathbb{R}^{\statedim_1} \times \ldots \times \mathbb{R}^{\statedim_\torusdim}, x = \left(x_1,  \ldots, x_\torusdim \right)^\tr
        $.
    \item For each $1 \leq  j \leq \torusdim$,  the ODE 
    \begin{equation} \label{eq:uncoupled_ODE}
        \dot{x}_j = \comp{F}{j}(x_j)
    \end{equation}
    has a periodic orbit $\Gamma_j$ which is \emph{normally hyperbolic}, i.e., it has an algebraically simple trivial Floquet multiplier and no other Floquet multipliers on the unit circle. We denote the period of $\Gamma_j$ by $T_j> 0$ and call $\omega_j = \frac{2 \pi}{T_j}$ the intrinsic frequency of oscillator~$j$. 
     \item The coupling strength $|\varepsilon| \ll 1$ is small. 
   \end{itemize}
\end{hyp}

The network dynamical system~\eqref{eq:IntroNet} is a particular example of an equation of the form~\eqref{eq:DDE}:
With $\tau = \max \{\tau_{j, k} \mid 1 \leq j, k \leq m \}$ the interaction term $G = (G_1, \dotsc, G_m)$ has components
\[
G_j(X(t,\cdot)) = \sum_{k=1}^\torusdim G_{j,k}\left(X_j(t,0), X_k(t,-\tau_{j,k})\right). 
\]

In Section~\ref{subsec:step1}--\ref{subsec:step4}, we prove the main Theorem~\ref{thm:intro} for such systems in four steps, summarizing each step at the beginning of the subsection. Our approach leads to a concrete algorithm to compute Taylor expansion the phase reduced vector field~$\ord{f}{\eps}$ (cf.~\eqref{eq:phase_intro}) together with the embeddings $\ord{e}{\eps}, \ord{E}{\eps}$ (cf.~\eqref{eq:embeddings_intro}); we summarize this algorithm in Section~\ref{subsec:algorithm}. 

\subsection{Step~1: Reformulation of the delay equation} \label{subsec:step1}

\begin{step}[cf.~Lemma~\ref{lem:ode-transport}]
    We re-formulate the delay differential equation
\[
    \dot{x}(t) = F(x(t)) + \epsilon G(X(t, \cdot)) \qquad \text{for}\ x\in \mathbb{R}^n \ \text{and} \ t \geq 0\, 
\]
(cf.~\eqref{eq:DDE}) as an ODE for the state variable $x(t)$ coupled to a transport equation for the history segment $X(t, s)$ 
\begin{align*} 
        \frac{d}{dt} x(t) &= F(x(t)) + \varepsilon G(X(t, 
        \, \cdot)) \\
        \frac{\partial}{\partial t} X(t, s)&= \frac{\partial}{\partial s} X(t, s)
\end{align*}
with $t \geq 0$, $s \in [-\tau, 0]$ and with boundary condition $X(t, 0) = x(t)$.   
\end{step}

\medskip

Observe that a smooth function
\[
X: \mathbb{R}_{\geq 0} \times [-\tau, 0], \qquad (t, s) \mapsto X(t, s)
\]
is of the form 
\[
X(t, s) = x(t +s) 
\]
if and only if it satisfies the transport equation 
\begin{subequations}
\begin{equation*} 
        \frac{\partial}{\partial t} X(t, s)= \frac{\partial}{\partial s} X(t, s)
\end{equation*}
together with the boundary condition 
\begin{equation*} 
   x(t) = X(t, 0). 
\end{equation*}
\end{subequations}
This  implies

\begin{subequations}
\begin{lemma} \label{lem:ode-transport}
    If~$x(t)$ is a smooth solution of the DDE~\eqref{eq:DDE}, then the pair $(x(t), X(t, s)) : = (x(t), x(t+s))$ is a solution of the coupled ODE-transport system
\begin{equation}
\label{eq:coupled_DDE}
\begin{aligned}
        \frac{d}{dt} x(t) &= F(x(t)) + \varepsilon G(X(t, 
        \, \cdot)),\\
        \frac{\partial}{\partial t} X(t, s) &= \frac{\partial}{\partial s} X(t, s)
\end{aligned}
    \tag{ode-tr.eq}
\end{equation}
for $t \geq 0$ and  $s \in [-\tau, 0]$
with boundary condition
\begin{equation} \label{eq:coupled_DDE_bc}
    x(t) = X(t, 0).  \tag{ode-tr.bc}
\end{equation}
Conversely, if $(x(t), X(t, s))$ is a smooth solution of the coupled system~\eqref{eq:coupled_DDE} with boundary condition~\eqref{eq:coupled_DDE_bc}, then $x(t)$ is a solution of the DDE~\eqref{eq:DDE}. 
\end{lemma}
\end{subequations}

In the above, we have (intentionally) not specified on which state space we consider the coupled ODE-transport system~\eqref{eq:coupled_DDE}; we will discuss this issue in Section~\ref{sec:discussion}.

\subsection{Step~2: Derivation of a conjugacy equation}
\begin{step}[cf.~Lemma~\ref{lem:ConjEq}]
We match the phase reduced dynamics on the invariant torus $\torus{m}$ to the dynamics of the coupled ODE-transport equation~\eqref{eq:coupled_DDE}--\eqref{eq:coupled_DDE_bc}. To do so, 
we parametrize the torus using two maps $\ord{e}{\eps}: \torus{m} \to \mathbb{R}^n$ and $\ord{E}{\eps}: \torus{m} \to C^1([-\tau, 0], \mathbb{R}^n)$; we show that this parametrization sends solutions of the flow on the torus (i.e., of the phase-reduced flow)
\[
\dot{\phi}(t) = \ord{f}{\eps}(\phi(t))
\]
to solutions of the coupled ODE-transport equation~\eqref{eq:coupled_DDE}--\eqref{eq:coupled_DDE_bc} if the \textbf{conjugacy equations}
\begin{align*}
    \partial_\varphi \ord{e}{\eps}(\varphi) \cdot \ord{f}{\eps}(\varphi) &= F(\ord{e}{\eps}(\varphi)) + \epsilon G(\ord{E}{\eps}(\varphi, 
    \, .)) \\
        \partial_\varphi \ord{E}{\eps}(\varphi, s) \cdot \ord{f}{\eps}(\varphi) &= \partial_s \ord{E}{\eps}(\varphi, s)
\end{align*}
(for $\phi \in \torus{m}$ and $s \in [-\tau, 0]$)    and the `boundary condition' $$\ord{e}{\eps}(\phi) = \ord{E}{\eps}(\phi, 0)$$ 
(for $\phi \in \torus{m}$) are satisfied.
\end{step}

As pre-empted in the introduction of this paper, we assume that the coupled ODE-transport system~\eqref{eq:coupled_DDE}--\eqref{eq:coupled_DDE_bc}
possesses a $\torusdim$-dimensional invariant torus, which we parametrize via a smooth map
\begin{equation} \label{eq:embedding}
    \torus{m} \ni \phi \mapsto (\ord{e}{\eps}(\phi), \ord{E}{\eps}(\phi, \, .)) \in \mathbb{R}^n \times C^1([-\tau, 0], \mathbb{R}^n)
\end{equation}
that satisfies the boundary condition
$
\ord{e}{\eps}(\phi) = \ord{E}{\eps}(\phi, 0)$.
(Note that, in particular, the map $(\phi, s) \mapsto (\ord{e}{\eps}(\phi), \ord{E}{\eps}(\phi, s))$ is now~$C^1$.). 
Since the torus is finite dimensional and invariant under the dynamics 
\eqref{eq:coupled_DDE}--\eqref{eq:coupled_DDE_bc}, the dynamics on the torus is conjugate to an ODE of the form 
\begin{equation} \label{eq:reduced_ODE}
    \dot{\varphi}(t) = \ord{f}{\eps}(\varphi(t))
\end{equation}
for some vector field 
\[ 
\ord{f}{\eps}: \torus{\torusdim} \to \mathbb{R}^\torusdim,
\]
and the ODE~\eqref{eq:reduced_ODE} is exactly the phase-reduced equation that we aim to compute. 

To do so, we look for an  embedding~\eqref{eq:embedding} that parametrizes the invariant torus, and sends solutions of the phase reduced equation~\eqref{eq:reduced_ODE} to solutions of the coupled ODE-transport system~\eqref{eq:coupled_DDE}-\eqref{eq:coupled_DDE_bc}. 
The latter means that 
\begin{equation} \label{eq:X_E}
    x(t) := \ord{e}{\eps}(\phi(t)), \qquad X(t, s) := \ord{E}{\eps}(\phi(t), s)
\end{equation}
should be a solution of the coupled transport-ODE system~\eqref{eq:coupled_DDE} whenever~$\phi(t)$ is a solution of the phase equation~\eqref{eq:reduced_ODE}. Differentiating the relation~\eqref{eq:X_E} with respect to~$t$ gives that 
\begin{align*}
    \dot{x}(t) &= \partial_\phi \ord{e}{\eps}(\phi(t)) \cdot \dot{\phi}(t) 
    = \partial_\phi \ord{e}{\eps}(\phi(t)) \cdot \ord{f}{\eps}(\phi(t)), \\
    \partial_t X(t, s) &= \partial_\phi \ord{E}{\eps}(\phi(t), s) \cdot \dot{\phi}(t) 
    = \partial_\phi \ord{E}{\eps}(\phi(t), s) \cdot \ord{f}{\eps}(\phi(t))
\end{align*}
But at the same time, we can compute the quantities $\dot{x}(t)$ and $\partial_t X(t, s)$ from the coupled ODE-transport system~\eqref{eq:coupled_DDE} as
\begin{align*}
    \dot{x}(t) &= F(\ord{e}{\eps}(\phi(t)) + \epsilon G(\ord{E}{\eps}(\varphi(t), 
    \, \cdot)), \\
    \partial_t X(t, s) &= \partial_s \ord{E}{\eps}(\varphi(t), s), 
\end{align*}
and from there we conclude  
\begin{lemma}\label{lem:ConjEq}
    Let 
\begin{align*}
    \ord{e}{\eps}: \torus{\torusdim} &\to \mathbb{R}^n, & \phi &\mapsto \ord{e}{\eps}(\phi) \\
    \ord{E}{\eps}: \torus{\torusdim} &\to C^\infty([-\tau,0], \mathbb{R}^n), & \phi &\mapsto \ord{E}{\eps}(\phi, \, \cdot )
\end{align*}
be smooth maps. Then the mapping  
\[
(x(t), X(t, s)) := (\ord{e}{\eps}(\phi(t)), \ord{E}{\eps}(\phi(t), s))
\]
sends solutions $\phi(t)$ of the phase equation~\eqref{eq:reduced_ODE} to solutions of the coupled ODE-transport sytem $(x(t), X(t, s))$ of~\eqref{eq:coupled_DDE}--\eqref{eq:coupled_DDE_bc} if and only if the \textbf{conjugacy equations}
    \begin{subequations}
    \begin{align} \label{eq:con_full_pt1}
        \partial_\varphi \ord{e}{\eps}(\varphi) \cdot \ord{f}{\eps}(\varphi) &= F(\ord{e}{\eps}(\varphi)) + \epsilon G(\ord{E}{\eps}(\varphi, 
    \, \cdot)) \tag{conj.fin}\\
        \label{eq:con_full_pt2} \tag{conj.tr}
        \partial_\varphi \ord{E}{\eps}(\varphi, s) \cdot \ord{f}{\eps}(\varphi) &= \partial_s \ord{E}{\eps}(\varphi, s)
    \end{align}
are satisfied all $\phi \in \torus{\torusdim}$ and $s \in [-\tau, 0]$ and the `boundary condition'
\begin{equation} \label{eq:con_full_bc}
    \ord{e}{\eps}(\phi) = \ord{E}{\eps}(\phi, 0) \tag{conj.bc}
\end{equation}
\end{subequations}
is satisfied for all $\phi \in \torus{\torusdim}$. 
\end{lemma}

Foreshadowing the notation in later sections, we denote by $\ord{f}{0}, \ord{e}{0}$ and $\ord{E}{0}$ functions that solve the conjugacy equations~\eqref{eq:con_full_pt1}--\eqref{eq:con_full_pt2} for $\epsilon = 0$, i.e., they satisfy 
\begin{equation} \label{eq:con_order_0}
    \begin{aligned}
        \partial_\varphi 
        \ord{e}{0}(\phi)\cdot \ord{f}{0}(\phi) &= F(\ord{e}{0}(\varphi)), \\
        \partial_\varphi \ord{E}{0}(\varphi, s) \cdot \ord{f}{0}(\varphi) &= \partial_s \ord{E}{0}(\varphi, s).
    \end{aligned}
\end{equation}
These equations are exactly satisfied when the functions~$\ord{e}{0}$ and~$\ord{E}{0}$ parametrize the unperturbed invariant torus of the uncoupled ODE $\dot{x} = F(x)$, and when they semi-conjugate the dynamics of $\ord{f}{0}$ on $\torus{\torusdim}$  to the dynamics of the ODE $\dot{x} = F(x)$. 
We will call these equations the \textbf{$0$th order homological equations}. 
The next lemma shows that we can choose the parametrization $\ord{e}{0}, \, \ord{E}{0}$ in such a way that the vector field $\ord{f}{0}$ on the unperturbed invariant torus is constant.
This means that the phases on the unperturbed torus rotate uniformly, which is the typical way to assign a phase to isolated oscillators. 
We will always choose this particular solution of the 0th-order equations~\eqref{eq:con_order_0} as the `base case' for the algorithm to compute the Taylor expansion of $\ord{f}{\eps}, \ord{e}{\eps}$ and~$\ord{E}{\eps}$ in Section~\ref{subsec:step4}.%

\begin{lemma} \label{lem:E0f0}
    For $1 \leq j \leq \torusdim$, let $\Gamma_j$ be the periodic orbit of the uncoupled ODE~\eqref{eq:uncoupled_ODE} with period $T_j > 0$ and frequency $\omega_j = \frac{2\pi}{T_j}$ (cf.~Hypothesis~\ref{hyp:system}). 
    Then the functions 
    \begin{equation} \label{eq:E0f0}
    \begin{aligned}
    \ord{e}{0}(\phi) & = \left( \Gamma_1 \left(\frac{\varphi_1}{\omega_1}\right), \ldots, \Gamma_n \left(\frac{\varphi_\torusdim}{\omega_\torusdim} \right) \right) \\
     \ord{E}{0}(\varphi, s) &= \left( \Gamma_1 \left(\frac{\varphi_1}{\omega_1}+s \right), \ldots, \Gamma_\torusdim \left(\frac{\varphi_\torusdim}{\omega_\torusdim}+s \right) \right) \\
     \ord{f}{0}(\varphi) &\equiv \left(\omega_1, \ldots, \omega_\torusdim\right)
    \end{aligned}
    \end{equation}
    solve the $0$th order homological equations~\eqref{eq:con_order_0} together with the relation $\ord{e}{0}(\phi) = \ord{E}{0}(\phi, 0)$. 
\end{lemma}
\begin{proof}
The definitions of $\ord{e}{0}$ and $\ord{E}{0}$ in~\eqref{eq:E0f0} directly imply that $\ord{e}{0}(\phi) = \ord{E}{0}(\phi, 0)$. 

We next verify that the quantities $\ord{e}{0}$, $\ord{E}{0}$ and $\ord{f}{0}$ defined in~\eqref{eq:E0f0} satisfy the second equality in~\eqref{eq:con_order_0}. To do so, we compute that 
\begin{equation} \label{eq:E0_der}
    \begin{aligned}
        \partial_\varphi \ord{E}{0}(\varphi, s) \cdot \ord{f}{0}(\varphi) &= \left( \frac{1}{\omega_1}\dot{\Gamma}_1\left(\frac{\varphi_1}{\omega_1} +s \right), \ldots, \frac{1}{\omega_n}\dot{\Gamma}_\torusdim\left(\frac{\varphi_\torusdim}{\omega_\torusdim} + s \right) \right) \cdot \left(\omega_1, \ldots, \omega_\torusdim\right) \\
        &=   \left( \dot{\Gamma}_1\left(\frac{\varphi_1}{\omega_1} + s \right), \ldots, \dot{\Gamma}_n\left(\frac{\varphi_\torusdim}{\omega_\torusdim} + s\right) \right).
    \end{aligned}
\end{equation}
But simultaneously it holds that 
\begin{align*}
    \partial_s \ord{E}{0}(\varphi, s) = \left( \dot{\Gamma}_1\left(\frac{\varphi_1}{\omega_1} + s \right), \ldots, \dot{\Gamma}_\torusdim\left(\frac{\varphi_\torusdim}{\omega_\torusdim} + s\right) \right)
\end{align*}
and hence $\partial_\varphi \ord{E}{0}(\varphi, s) \cdot \ord{f}{0}(\varphi) = \partial_s \ord{E}{0}(\varphi, s)$, i.e., the second equality in~\eqref{eq:con_order_0} is satisfied. 

We finally verify that the quantities $\ord{e}{0}$, $\ord{E}{0}$ and $\ord{f}{0}$ defined in~\eqref{eq:E0f0} satisfy the first equality in~\eqref{eq:con_order_0}. Since~$\Gamma_j$ is a solution of the uncoupled ODE~\eqref{eq:uncoupled_ODE}, it holds that 
\[ 
\comp{F}{j}(\Gamma_j(t)) = \dot{\Gamma}_j(t),
\]
or, with $t= \frac{\varphi_j}{\omega_j}$, 
\[
\comp{F}{j}\left( \Gamma_j\left(\frac{\varphi_j}{\omega_j}\right) \right) = \dot{\Gamma}_j\left(\frac{\varphi_j}{\omega_j}\right).
\]
Since by assumption~$F(x)$ is of the form $F(x_1, \ldots, x_m) = (F_1(x_1), \ldots, F_m(x_m))$ (cf.~Hypothesis~\ref{hyp:system}), this implies that
\begin{align*}
    F(\ord{e}{0}(\varphi)) &= \left( \comp{F}{1} \left(\Gamma_1\left(\frac{\varphi_1}{\omega_1}\right) \right), \ldots, \comp{F}{m} \left(\Gamma_\torusdim\left(\frac{\varphi_\torusdim}{\omega_\torusdim}\right) \right) \right) \\
    &=  \left( \dot{\Gamma}_1\left(\frac{\varphi_1}{\omega_1}\right), \ldots, \dot{\Gamma}_\torusdim\left(\frac{\varphi_\torusdim}{\omega_\torusdim}\right) \right).
\end{align*}
The last expression is equal to $\partial_\varphi \ord{E}{0}(\varphi, 0) \cdot \ord{f}{0}(\varphi)$ (cf.~\eqref{eq:E0_der}) and hence it follows that $\partial_\varphi \ord{E}{0}(\varphi, 0) \cdot \ord{f}{0}(\varphi) = F(\ord{e}{0}(\varphi))$. But we had already observed that $\ord{E}{0}(\phi, s) = \ord{e}{0}(\phi)$, and hence we conclude that $\partial_\varphi \ord{e}{0}(\varphi) \cdot \ord{f}{0}(\varphi) = F(\ord{e}{0}(\varphi))$, as claimed. 
\end{proof}

\subsection{Step~3: Derivation of order-by-order homological equations} \label{sec:iterative_con}

\begin{step}[cf.~Lemma~\ref{lem:con_order_j}]
We formally expand the functions~$\ord{e}{\eps}, \ord{E}{\eps}$ and~$\ord{f}{\eps}$ in the small parameter~$\epsilon$ as
\begin{equation*} 
\begin{aligned}
    \ord{f}{\eps} &= \ord{f}{0} + \varepsilon \ord{f}{1} +\dotsb + \eps^\el \ord{f}{\el} +\dotsb \\    \ord{e}{\eps} & = \ord{e}{0} + \varepsilon \ord{e}{1} +\dotsb + \epsilon^\el \ord{e}{\el} +\dotsb \\
    \ord{E}{\eps} &= \ord{E}{0} + \varepsilon \ord{E}{1} +\dotsb + \eps^\el \ord{E}{\el} +\dotsb
\end{aligned}
\end{equation*}
and show that the conjugacy equations~\eqref{eq:con_full_pt1}--\eqref{eq:con_full_pt2} are satisfied up to order $\El$ if and only if for each $1 \leq \el \leq \El$, the functions $\ord{f}{\el}, \ord{e}{\el}$ and $\ord{E}{\el}$ solve the \textbf{$\el$th order homological equations}
\begin{align*}
    \partial_{\omega} \ord{e}{\el}(\varphi) - F'(\ord{e}{0}(\varphi)) \cdot \ord{e}{\el}(\varphi) + \partial_\varphi \ord{e}{0}(\varphi) \cdot \ord{f}{\el}(\varphi) &= \ord{\eta}{\el}(\varphi)   \\ 
       \partial_{\omega} \ord{E}{\el}(\varphi, s) - \partial_s \ord{E}{\el}(\varphi, s)  + \partial_\varphi \ord{E}{0}(\varphi, s) \cdot \ord{f}{\el}(\varphi) &= \ord{H}{\el}(s, \varphi), 
\end{align*}
Here, $\omega := (\omega_1, \ldots, \omega_n)$ is the vector consisting of the frequencies of the periodic orbits in the uncoupled system, and we use the notations $
\partial_\omega \ord{e}{\el}(\varphi): = \partial_\phi \ord{e}{\el}(\phi) \cdot \omega$  and $\partial_\omega \ord{E}{\el}(\varphi, s) := \partial_\phi \ord{E}{\el}(\varphi, s) \cdot \omega$.
The functions~$\ord{\eta}{\el}, \ord{H}{\el}$ importantly only depend on `lower order information,' i.e., functions $\ord{f}{\el'}, \ord{e}{\el'}$ and $\ord{E}{\el'}$ with $\el' < \el$, and hence the right hand side of the $\el$th homological equation can iteratively be computed once the lower order equations have been solved.
\end{step}

\begin{lemma} \label{lem:con_order_j}
Let 
\begin{align*}
    \ord{f}{\eps}: \torus{\torusdim} &\to \mathbb{R}^n, & \phi &\mapsto \ord{f}{\eps}(\phi) \\
    \ord{e}{\eps}: \torus{\torusdim} &\to \mathbb{R}^n, & \phi &\mapsto \ord{e}{\eps}(\phi) \\
    \ord{E}{\eps}: \torus{\torusdim} &\to C^\infty([-\tau, 0], \mathbb{R}^n), & \phi &\mapsto \ord{E}{\eps}(\phi, \, \cdot)
\end{align*}
be smooth maps that depend smoothly on $\epsilon$ and that satisfy the conjugacy equations~\eqref{eq:con_full_pt1} together with the boundary relation~\eqref{eq:con_full_bc}. We formally expand the maps $\ord{f}{\eps}, \ord{e}{\eps}, \ord{E}{\eps}$ in the small parameter $\epsilon$ as \begin{equation} \label{eq:expansion_E_f}
\begin{aligned}
    \ord{f}{\eps} &= \ord{f}{0} + \varepsilon \ord{f}{1} +\dotsb + \eps^\el \ord{f}{\el} +\dotsb \\
    \ord{e}{\eps} & = \ord{e}{0} + \varepsilon \ord{e}{1} +\dotsb + \eps^\el \ord{e}{\el} +\dotsb \\
    \ord{E}{\eps} &= \ord{E}{0} + \varepsilon \ord{E}{1} +\dotsb + \eps^\el \ord{E}{\el} +\dotsb    
\end{aligned}
\end{equation}
and denote by 
\begin{equation}
    \label{eq:omega}
\omega := (\omega_1, \ldots, \omega_\torusdim)
\end{equation}
the vector whose entries are the frequencies of the periodic orbits in the uncoupled system. 

Then the maps $\ord{e}{0}, \ord{E}{0}, \ord{f}{0}$ satisfy 
the $0$th order conjugacy equations~\eqref{eq:con_order_0}. 
Suppose we solve the $0$th order equations in such a way that $\ord{f}{0}(\phi) \equiv \omega$ (cf.~Lemma~\ref{lem:E0f0}). Then the conjugacy equation~\eqref{eq:con_full_pt1}--\eqref{eq:con_full_pt2} is satisfied up to order $\El \in \mathbb{N}$ if and only if the maps $\ord{e}{\el}, \ord{E}{\el}, \ord{f}{\el}$ for $1 \leq \el \leq \El$ satisfy the \textbf{$\el$th order homological equations}
\begin{subequations}
    \begin{align} \label{eq:con_order_j_n1}
        \partial_{\omega} \ord{e}{\el}(\varphi) - F'(\ord{e}{0}(\varphi)) \cdot \ord{e}{\el}(\varphi) + \partial_\varphi \ord{e}{0}(\varphi) \cdot \ord{f}{\el}(\varphi) &= \ord{\eta}{\el}(\varphi)  \tag{hom.fin} \\ 
       \partial_{\omega} \ord{E}{\el}(\varphi, s) - \partial_s \ord{E}{\el}(\varphi, s)  + \partial_\varphi \ord{E}{0}(\varphi, s) \cdot \ord{f}{\el}(\varphi) &= \ord{H}{\el}(s, \varphi), \tag{hom.tr}
       \label{eq:con_order_j_n2}
    \end{align}
\end{subequations}
for all $\phi \in \mathbb{T}^\torusdim$ and $s \in [-\tau, 0]$. 
Here, the inhomogeneous right hand side $\ord{\eta}{\el}$ depends only on $\ord{f}{0}, \ldots, \ord{f}{\el-1}$, $\ord{e}{0}, \ldots, \ord{e}{\el-1}$ and $\ord{E}{0}, \ldots, \ord{E}{\el-1}$; the term $\ord{H}{\el}$ depends only on  $\ord{f}{0}, \ldots, \ord{f}{\el-1}$ and $\ord{E}{0}, \ldots, \ord{E}{\el-1}$; 
and we have used the notation 
\[
\partial_\omega \ord{e}{\el}(\varphi): = \partial_\phi \ord{e}{\el}(\phi) \cdot \omega, \qquad \partial_\omega \ord{E}{\el}(\varphi, s) := \partial_\phi \ord{E}{\el}(\varphi, s) \cdot \omega .
\]
\end{lemma}
\begin{proof} We first derive the first equality in~\eqref{eq:con_order_0} together with the equality~\eqref{eq:con_order_j_n1} from the conjugacy equation~\eqref{eq:con_full_pt1}. 
Substituting the expansion~\eqref{eq:expansion_E_f} in the left hand side of~\eqref{eq:con_full_pt1} yields
\begin{equation} \label{eq:expansion_lhs}
(\partial_\varphi \ord{e}{0}(\phi) +\dotsb +  \varepsilon^{\el} \partial_\varphi\ord{e}{\el}(\phi) +\dotsb  ) \cdot (\ord{f}{0}(\varphi) +\dotsb + \varepsilon^{\el} \ord{f}{\el}(\varphi) +\dotsb )    
\end{equation}
of which the constant term in~$\epsilon$ is given by 
\[ 
\partial_\varphi \ord{e}{0}(\phi) \cdot f_0(\varphi)
\]
and for $\el \geq 1$, the term of order $\epsilon^\el$ is given by 
\[
\partial_\phi \ord{e}{0}(\phi) \cdot \ord{f}{\el}(\phi) + \partial_\phi \ord{e}{\el}(\phi) \cdot \ord{f}{0}(\phi) +  \sum_{\el'=1}^{\el-1} \partial_\phi \ord{e}{\el'}(\phi) \cdot \ord{f}{\el-\el'}(\phi).
\]
Substituting the expansion~\eqref{eq:expansion_E_f} in the right hand side of~\eqref{eq:con_full_pt1} yields 
\begin{equation} \label{eq:expansion_rhs}
    F\left(\ord{e}{0}(\phi) + \varepsilon \ord{e}{1}(\phi) +\dotsb \right) + \epsilon G\left(\ord{E}{0}(\phi, \, .) + \epsilon\ord{E}{1}(\phi, \, .) +\dotsb \right). 
\end{equation}
The  constant term in $\varepsilon$ in~\eqref{eq:expansion_rhs} is given by 
\[
F(\ord{e}{0}(\phi)),
\]
and equating this to the constant term in~\eqref{eq:expansion_lhs} gives the first equality in~\eqref{eq:con_order_0}. 
For $\el \geq 1$, the term of order $\epsilon^\el$ is given by 
\begin{equation} \label{eq:expansion_bell}
    \begin{aligned}
        F'(\ord{e}{0}(\phi)) \ord{e}{\el}(\phi) + \sum_{\el' = 2}^\el B_{j, k} \left[D^{(\el')} F(\ord{e}{0}(\phi)), \ord{e}{1}(\phi), \ldots, \ord{e}{\el-\el'+1}(\phi) \right]  \\ + \sum_{\el' = 1}^{\el-1}C_{\el, \el'}\left[D^{(\el')} G(\ord{E}{0}(\phi, \, .)), \ord{E}{0}(\phi, \, .), \ldots, \ord{E}{\el-\el'}(\phi, \, .) \right].
    \end{aligned}
\end{equation}
Here the functions~$B_{\el, \el'}$ and~$C_{\el, \el'}$ represent combinatorial expressions that can be explicitly derived from the Faa di Bruno formula~\cite{Constantine1996};
we refrain from stating the explicit expression here, but only point out that they do \emph{not} depend on $\ord{e}{\el'}, \, \ord{E}{\el'}$ for $\el' \geq \el$. We equate~\eqref{eq:expansion_bell} to the term of order $\epsilon^\el$ in~\eqref{eq:expansion_lhs} and write the resulting equality as
\begin{align*}
    \partial_\varphi \ord{e}{\el}(\varphi)\cdot \ord{f}{0}(\varphi) - F'(\ord{e}{0}(\phi)) \cdot \ord{e}{\el}(\varphi) + \partial_\varphi \ord{e}{0}(\varphi) \cdot \ord{f}{\el}(\varphi) = \ord{\eta}{\el}(\phi)
\end{align*}
where $\ord{\eta}{\el}$ depends on $\ord{f}{0}, \ldots, \ord{f}{\el-1}$, $\ord{e}{0}, \ldots, \ord{e}{\el-1}$ and $\ord{E}{0}, \ldots, \ord{E}{\el-1}$; substituting the solution $\ord{f}{0}(\phi) \equiv \omega$ obtained in Lemma~\ref{lem:E0f0} then yields~\eqref{eq:con_order_j_n1}.

\medskip

We next derive the second equality in~\eqref{eq:con_order_0} and the equality~\eqref{eq:con_order_j_n2} from the conjugacy equation~\eqref{eq:con_full_pt2}. Substituting the expansion~\eqref{eq:expansion_E_f} in the left hand side of~\eqref{eq:con_full_pt2} gives
\begin{equation} \label{eq:expansion_lhs_pde}
    (\partial_\varphi \ord{E}{0}(\varphi, s) +\dotsb +  \varepsilon^{\el} \partial_\varphi \ord{E}{\el}(\varphi, s) +\dotsb  ) \cdot (\ord{f}{0}(\varphi) +\dotsb + \varepsilon^{\el} \ord{f}{\el}(\varphi) +\dotsb ) 
\end{equation}
of which the constant term is given by 
\[
\partial_\varphi \ord{E}{0}(\varphi, s) \cdot \ord{f}{0}(\phi);
\]
and equating this to the constant term in $\epsilon$ in the expression $\partial_s E(\phi, s) = \partial_s \ord{E}{0}(\phi, s) + \epsilon \ord{E}{1}(\phi, s) +\dotsb$ gives the second equality in~\eqref{eq:con_order_0}. For $\el \geq 1$, the term of order $\eps^\el$ in~\eqref{eq:expansion_lhs_pde} is given by 
\[
\partial_\phi \ord{E}{0}(\phi, s) \cdot \ord{f}{\el}(\phi) + \partial_\phi \ord{E}{\el}(\phi, s) \cdot \ord{f}{0}(\phi) +  \sum_{k = 1}^{j-1} \partial_\phi \ord{E}{k}(\phi, s) \cdot \ord{f}{\el-k}(\phi). 
\]
We equate this to the term of order $\eps^\el$ in the expression $\partial_s E(\phi, s) = \partial_s \ord{E}{0}(\phi, s) + \epsilon \ord{E}{1}(
\phi, s) +\dotsb$
and write the resulting equality as 
\[
\partial_\varphi \ord{E}{\el}(\varphi, s) \cdot \ord{f}{0}(\varphi) - \partial_s \ord{E}{\el}(\varphi, s)  + \partial_\varphi \ord{E}{0}(\varphi, s) \cdot \ord{f}{\el}(\varphi) = \ord{H}{\el}(\phi, s),
\]
where the value of~$\ord{H}{\el}$ depends on $\ord{f}{0}, \dotsc, \ord{f}{\el-1}$ and on $\ord{E}{0}, \dotsc, \ord{E}{\el-1}$. Substituting~$\ord{f}{0}(\phi) \equiv \omega$ (Lemma~\ref{lem:E0f0}) then yields~\eqref{eq:con_order_j_n2}.
\end{proof}

\begin{remark} \label{rmk:rhs_explicit}
The functions $\ord{\eta}{1}, \ord{\eta}{2}$ are explicitly given by
\begin{equation} \label{eq:G1G2_expl}
    \begin{aligned}
        \ord{\eta}{1}(\phi) &= G(\ord{E}{0}(\phi, \, . )) \\
        \ord{\eta}{2}(\phi) &= \frac{1}{2} F''(\ord{e}{0}(\phi))[\ord{e}{1}(\phi), \ord{e}{1}(\phi)] -  \partial_\phi \ord{e}{1}(\phi) \cdot \ord{f}{1}(\phi) \\ 
        & \qquad + G'(\ord{E}{0}(\phi, \, .))\ord{E}{1}(\phi, \, .) 
    \end{aligned}
\end{equation}
and the functions $H_1, H_2$ are explicitly given by 
\begin{equation} \label{eq:H1H2_expl}
    \begin{aligned}
        \ord{H}{1}(\phi, s) & \equiv 0 \\
        \ord{H}{2}(\phi, s) &= - \partial_{\phi} \ord{E}{1}(\phi, s) \cdot \ord{f}{1}(\phi). 
    \end{aligned}
\end{equation}
\end{remark}

\subsection{Step~4: Solving the homological equations} \label{subsec:step4}

\begin{step}
    The fourth step is to show that the $\el$th order homological equations~\eqref{eq:con_order_j_n1}--\eqref{eq:con_order_j_n2} derived in Step~3 are solvable for each $\el \in \mathbb{N}$. We proceed as follows.
        
        First (cf.~Lemma~\ref{lem:N_L} and Lemma~\ref{lem:con_order_splitted}), we write the embedding function $\ord{e}{\eps}: \torus{\torusdim} \to \mathbb{R}^n$ of the invariant torus 
        in a suitable coordinate frame based on the unperturbed torus.
        Define the map $\Tg: \torus{m} \to \mathbb{R}^{m \times m}, \phi\mapsto\partial_\phi \ord{e}{0}(\phi)$, and note that, for each $\phi \in \torus{m}$, the columns of~$\Tg(\phi)$ form a basis for the tangent space to the unperturbed torus at the point $\ord{e}{0}(\phi)$.
        Next, we construct (cf.~Lemma~\ref{lem:N_L}) a matrix valued function $N: \torus{m} \to \mathbb{R}^{(n-m) \times m}$ whose columns span a certain normal space to the unperturbed torus. 
        This allows to write the $\el$th order Taylor coefficient~$\ord{e}{\el}$ of the embedding function~$\ord{e}{\eps}$ as a direct sum
        \[
        \ord{e}{\el}(\phi) = \Tg(\phi) \ord{g}{\el}(\phi) + N(\phi) \ord{h}{\el}(\phi),
        \]
        where the function~$\ord{g}{\el}$ determines the component of~$\ord{e}{\el}$ tangential to the unperturbed invariant torus and the function~$\ord{h}{\el}$ determines the component of~$\ord{e}{\el}$ normal to it.
        We show (Lemma~\ref{lem:con_order_splitted}) that this coordinate transformation splits~\eqref{eq:con_order_j_n1} into tangential and normal directions leading to the \textbf{split $\el$th order homological equations}       
        \begin{align*}
        \Tg(\phi) [\partial_{\omega} \ord{g}{\el}(\phi)  + \ord{f}{\el}(\phi)] &= \pi(\phi) \ord{\eta}{\el}(\phi) \\
         N(\phi)[\partial_{\omega} \ord{h}{\el}(\phi) - L \ord{h}{\el}(\phi)] &= (I-\pi(\phi))[\ord{\eta}{\el}(\phi)],
        \end{align*}
        where~$\pi(\phi)$ denotes the projection on the tangent space of the invariant torus along the normal space and $I$~the identity.

    Second (cf.~Lemmas~\ref{lem:SolTang} and~\ref{lem:SolNorm}), we prove that the split $\el$th order homological equations are solvable for $\ord{g}{\el}$, $\ord{h}{\el}$ and $\ord{f}{\el}$, and provide explicit equations for the solutions in terms of Fourier coefficients. Using again the coordinate transformation $\ord{e}{\el}(\phi) = \Tg(\phi) \ord{g}{\el}(\phi) + N(\phi) \ord{h}{\el}(\phi)$, we obtain solutions~$\ord{e}{\el}$ and~$\ord{f}{\el}$ for the $\el$th order homological equation~\eqref{eq:con_order_j_n1} from Step~3.

        Third (cf.~Lemma~\ref{lem:TranspEq}), we show that the $\el$th order homological equation~\eqref{eq:con_order_j_n2} from Step~3, which corresponds to the transport part of the homological equations, is solvable. 
\end{step}

Together, this leads to a concrete algorithm to solve the $\el$th~order homological equations~\eqref{eq:con_order_j_n1}--\eqref{eq:con_order_j_n2} order by order, and thereby obtain an approximation of the vector field $\ord{f}{\eps}$ 
 and the  embeddings  $\ord{e}{\eps}$ and $\ord{E}{\eps}$ up to arbitrary accuracy in~$\epsilon$. 
We summarize this algorithm in Section~\ref{subsec:algorithm}.

\subsubsection{Suitable coordinates}

The function $\ord{e}{0}(\phi): \torus{m} \to \mathbb{R}^n$ defined in~\eqref{eq:E0f0} parametrizes the unperturbed invariant torus in the uncoupled system $\dot{x} = F(x)$; consequently, 
given a $\phi \in \torus{m}$, the columns of the matrix~$\Tg(\phi) = \partial_\phi \ord{e}{0}(\phi)$ form a basis for the tangent space to the unperturbed torus at the point $\ord{e}{0}(\phi)$.
The next lemma constructs a matrix-valued map $\torus{m} \ni \phi \mapsto N(\phi)$ such that, for each $\phi \in \torus{m}$, the vectors of $N(\phi)$ span a certain \emph{normal} space to the unperturbed invariant torus. Together, $T$ and $N$ define a coordinate frame that is particularly suitable to study the $\el$th order homological equation~\eqref{eq:con_order_j_n1} in Lemma~\ref{lem:con_order_splitted}. 

The proof of the lemma below is based on Floquet theory; here, we give an outline of this proof, but refer to~\cite{VonderGracht2023a} for a complete proof.  

\begin{lemma} \label{lem:N_L}
Let $F: \mathbb{R}^n \to \mathbb{R}^n$ be the vector field of the uncoupled system defined in Hypothesis~\ref{hyp:system}. Let $\omega = (\omega_1, \ldots, \omega_\torusdim)$ be the vector consisting of the frequencies of the periodic orbits in the uncoupled system (cf.~\eqref{eq:omega}) and let
$\ord{e}{0}: \torus{\torusdim} \to \mathbb{R}^n$ be the embedding of the unperturbed invariant torus in the uncoupled system defined in~\eqref{eq:E0f0}. 
Define
\[
\Tg: \torus{m} \to \R^{m \times m}, \qquad T(\phi) := \partial_\phi \ord{e}{0}(\phi)\, .
\]
For each $\phi \in \torus{m}$, the columns of $\Tg(\phi)$ form a basis for the tangent space to unperturbed torus at the point~$\ord{e}{0}(\phi)$.

Then there exists a smooth matrix-valued map
\[
\begin{aligned}
    N: \torus{\torusdim} &\to \mathcal{L}(\mathbb{R}^{\statedim-\torusdim}, \mathbb{R}^\statedim) , \qquad \phi \mapsto N(\phi)
\end{aligned}
\]
together with a hyperbolic matrix
\[
L \in \mathbb{R}^{(\statedim - \torusdim) \times (\statedim - \torusdim)}
\]
such that the following two statements hold: 
\medskip
\begin{enumerate}
    \item For each $\phi \in \torus{\torusdim}$, the image of the matrix $N(\phi): \mathbb{R}^{n -m} \to \mathbb{R}^n$ is transverse to the image of the matrix $\Tg(\phi): \mathbb{R}^m \to \mathbb{R}^n$, i.e.,
    \begin{equation} \label{eq:decomposition_e_N}
    \im \Tg(\phi)   \oplus \im N(\varphi) = \mathbb{R}^{\statedim}; 
    \end{equation}
    \item The map $\phi \mapsto N(\phi)$ satisfies the differential equation
  \begin{equation} \label{eq:pde_N}
        \partial_{\omega} N(\phi) + N(\phi) L = F'(\ord{e}{0}(\phi)) N(\phi).
  \end{equation}
    
\end{enumerate}
\end{lemma}
\begin{proof}[Outline of the proof]
By assumption (cf.~Hypothesis~\ref{hyp:system}), the function $F$ is of the form $F(x) = (F_1(x_1), \ldots, F_n(x_n))$ and each ODE $\dot{x}_j = F(x_j)$ has a hyperbolic periodic solution $\Gamma_j(t)$. Floquet's theorem (see, for example,~\cite{hale2009}) tells that the fundamental solution $\Phi_j(t)$ of the linearized system 
\begin{equation} \label{eq:linearised_ODE_j}
\dot{y}_j(t) = \comp{F}{j}'(\comp{\Gamma}{j}(t)) y_j(t)
\end{equation}
can be decomposed as 
\[
\Phi_j(t) = P_j(t) e^{B_j t},
\]
where the matrix-valued function $P_j(t)$ is~$T_j$-periodic and~$B_j$ is a constant matrix. 

Since~$\dot{\Gamma}_j(t)$ is a $T_j$-periodic solution of the linearized system~\eqref{eq:linearised_ODE_j}, the matrix~$B_j$ can be chosen in such a way that it has an eigenvalue 0; for this choice of $B_j$ it holds that $P_j(t) \Phi_j(0) = \Phi_j(t)$. Moreover, because we have assumed that the periodic solution $\Gamma_j$ is hyperbolic, the matrix $B_j$ has no other eigenvalues on the imaginary axis. 

We next choose a basis of $\mathbb{R}^{n_j}$ in which the first column of the map $P_j$ is equal to $\dot{\Gamma}_j$ and the first column of $B_j$ is equal to the zero column. 
For each $\phi_j \in \mathbb{R}/2 \pi Z$, we denote by 
\[
N_j(\phi_j): \mathbb{R}^{\statedim_j-1} \to \mathbb{R}^{\statedim_j}
\]
the matrix that we obtain by, in the chosen basis, removing the first column from $P_j\left(\frac{\phi_j}{\omega_j}\right)$
and denote by 
\[
L_j: \mathbb{R}^{\statedim_j - 1} \to \mathbb{R}^{\statedim_j - 1}
\]
the matrix we obtain by, in the chosen basis, removing the first column and the first row from $B_j$. This matrix $L_j$ is hyperbolic, and moreover one can prove that the map $\phi_j \mapsto N_j(\phi_j)$ satisfies the differential equation
\[
N_j'(\phi_j) \cdot \omega_j + N_j(\phi_j) L_j = F'(\Gamma_j(\phi)) N_j(\phi_j)
\]
(see~\cite{VonderGracht2023a} for a more precise definition of $N_j$ and $B_j$ and a more precise derivation of the above equality).  
By defining
\[
N(\phi) = (N_1(\phi_1), \ldots, N_\torusdim(\phi_\torusdim)) \qquad \mbox{and} \qquad L = \mbox{diag}(L_1, \ldots, L_\torusdim)
\]
the statement of the lemma follows. 
\end{proof}

We next consider the function~$\ord{e}{\el}$, which corresponds to the $\el$th order Taylor coefficient of torus embedding~$\ord{e}{\eps}$ (cf.~\eqref{eq:expansion_E_f}), and use Lemma~\ref{lem:N_L} to write $\ord{e}{\el}$ as
\[
\ord{e}{\el}(\phi) = \Tg(\phi) \ord{g}{\el}(\phi) + N(\phi) \ord{h}{\el}(\phi). 
\]
Here we think of the map $\ord{g}{\el}: \torus{\torusdim} \to \mathbb{R}^\statedim$ as determining the component of $\ord{e}{\el}$ along the tangent bundle of the unperturbed torus; and we think of the map $\ord{h}{\el}: \torus{\torusdim} \to \mathbb{R}^{\statedim-\torusdim}$ as determining the component of $\ord{e}{\el}$ normal to the tangent bundle of the unperturbed torus. 
The following lemma shows that this Ansatz transforms the $\el$th order homological equation~\eqref{eq:con_order_j_n1} into two decoupled equations, namely one equation for $\ord{g}{\el}$ along the tangent bundle and one equation for $\ord{h}{\el}$ normal to the tangent bundle. 
We point out that the $\el$th order contribution to the vector field on the torus, denoted by $\ord{f}{\el}$, only appears in the equation along the tangent bundle, in agreement with the fact that the vector field takes values in the tangent bundle. 

\begin{lemma} \label{lem:con_order_splitted}
Let $\omega = (\omega_1, \ldots, \omega_m)$ be the vector consisting of the frequencies of the periodic orbits in the uncoupled system (cf.~\eqref{eq:omega}) and let $\ord{e}{0}: \torus{\torusdim} \to \mathbb{R}^n$ be the embedding of the unperturbed invariant torus in the uncoupled system defined in~\eqref{eq:E0f0}. 
Moreover, define the matrix-valued function $\Tg(\phi) = \partial_\phi \ord{e}{0}(\phi)$, and denote by $\phi \mapsto N(\phi)$ be the matrix-valued function from Lemma~\ref{lem:N_L} whose columns span the normal bundle of the unperturbed invariant torus in the uncoupled system. 

Then the Ansatz
\begin{equation} \label{eq:ej_decomp}
    \ord{e}{\el}(\phi) = \Tg(\phi) \ord{g}{\el}(\phi) + N(\phi) \ord{h}{\el}(\phi)
\end{equation}
with $\ord{g}{\el}: \torus{\torusdim} \to \mathbb{R}^\torusdim$ and $\ord{h}{\el}: \torus{\torusdim} \to \mathbb{R}^{\statedim-\torusdim}$ transforms the $\el$th order homological equation~\eqref{eq:con_order_j_n1} into the \textbf{split $\el$th order homological equations}%
\begin{subequations}\label{eq:conj_order}
\begin{align}
    \Tg(\phi) [\partial_{\omega} \ord{g}{\el}(\phi)  + \ord{f}{\el}(\phi)] &= \pi(\phi) \ord{\eta}{\el}(\phi),\label{eq:conj_tangent} \tag{hom.fin.tangent}\\
    N(\phi)[\partial_{\omega} \ord{h}{\el}(\phi) - L \ord{h}{\el}(\phi)] &= (I-\pi(\phi))[\ord{\eta}{\el}(\phi)]. \label{eq:conj_normal} \tag{hom.fin.normal}
\end{align}
\end{subequations}
Here, for each $\phi \in \torus{\torusdim}$, $\pi(\phi): \mathbb{R}^n \to \mathbb{R}^n$ is the projection operator on the tangent bundle of the unperturbed torus along the normal bundle of the unperturbed torus, i.e., it satisfies
\begin{equation} \label{eq:projection}
    \pi(\phi) \Tg(\phi)v = \Tg(\phi)v \quad \mbox{and} \quad \pi(\phi) N(\phi) w = 0 
\end{equation}
for all $v \in \mathbb{R}^\torusdim$ and all $w \in \mathbb{R}^{\statedim-\torusdim}$. 
\end{lemma}

\begin{proof} 
We first compute the directional derivative of the Ansatz~\eqref{eq:ej_decomp} in the direction $\omega$ and find that 
\begin{equation} \label{eq:ej_der}
\begin{aligned} 
    \partial_{\omega} \ord{e}{\el}(\phi) &= \partial_\phi^2 \ord{e}{0}(\phi) [\omega, \ord{g}{\el}(\phi)] + \partial_\phi \ord{e}{0}(\phi) [\partial_{\omega} \ord{g}{\el}(\phi)] 
    \\&\qquad
    + \partial_{\omega} N(\phi)[\ord{h}{\el}(\phi)] + N(\phi)[\partial_{\omega} \ord{h}{\el}(\phi)].
\end{aligned}
\end{equation}
Substituting the Ansatz~\eqref{eq:ej_decomp} together with its derivative~\eqref{eq:ej_der} into the $\el$th order homological equation~\eqref{eq:con_order_j_n1} gives that 
\begin{equation} \label{eq:con_gjhj_interm1}
\begin{aligned}
    \partial_\phi^2 \ord{e}{0}(\phi) [\omega, \ord{g}{\el}(\phi)] + \partial_\phi \ord{e}{0}(\phi) [\partial_{\omega} \ord{g}{\el}(\phi)] + \partial_{\omega} N(\phi) [\ord{h}{\el}(\phi)] &\\ + N(\phi)[\partial_{\omega} \ord{h}{\el}(\phi)]  - F'(\ord{e}{0}(\phi)) \partial_\phi \ord{e}{0}(\phi) \ord{g}{\el}(\phi) & \\ - F'(\ord{e}{0}(\phi)) N(\phi) \ord{h}{\el}(\phi) + \partial_\phi \ord{e}{0}(\varphi) \ord{f}{\el}(\phi) &= \ord{\eta}{\el}(\phi).
\end{aligned}
\end{equation}
Since $\partial_\phi \ord{e}{0}(\phi) \cdot \omega = F(\ord{e}{0}(\phi))$ (cf.~\eqref{eq:con_order_0}), it holds that
\[
\partial_\phi^2 \ord{e}{0}(\phi) [\omega, v] = F'(\ord{e}{0}(\phi)) [\partial_\phi \ord{e}{0}(\phi) \cdot v]
\]
for all $v \in \mathbb{R}^N$, and hence in particular that 
\begin{equation} \label{eq:simplify_step1}
    \partial_\phi^2 \ord{e}{0}(\phi) [\omega, \ord{g}{\el}(\phi)] = F'(\ord{e}{0}(\phi)) [\partial_\phi \ord{e}{0}(\phi) \cdot \ord{g}{\el}(\phi)].
\end{equation}
Additionally, the differential equation for $N(\phi)$ in~\eqref{eq:pde_N} implies that 
\begin{equation} \label{eq:simplify_step2}
    \partial_{\omega} N(\phi) [\ord{h}{\el}(\phi)] + N(\phi) L \ord{h}{\el}(\phi) = F'(\ord{e}{0}(\phi)) N(\phi) \ord{h}{\el}(\phi).
\end{equation}
Using~\eqref{eq:simplify_step1}--\eqref{eq:simplify_step2}, we simplify~\eqref{eq:con_gjhj_interm1} to
\begin{equation} \label{eq:con_gjhj_interm2}
\begin{aligned}
    \partial_\phi \ord{e}{0}(\phi) [\partial_{\omega} \ord{g}{\el}(\phi)] + N(\phi)[\partial_{\omega} \ord{h}{\el}(\phi)]  &\\{} -  N(\phi)L\ord{h}{\el}(\phi) + \partial_\phi \ord{e}{0}(\varphi) \ord{f}{\el}(\phi) &= \ord{\eta}{\el}(\phi).
    \end{aligned}
\end{equation}
Next denote by $\pi(\phi): \mathbb{R}^\statedim \to \mathbb{R}^\statedim$ the projection operator onto the tangent bundle of the unperturbed torus along the normal bundle of the unperturbed torus defined in~\eqref{eq:projection}; then applying $\pi(\phi)$ to both sides of equation~\eqref{eq:con_gjhj_interm2} yields the first equation~\eqref{eq:conj_tangent}; and applying $I-\pi(\phi)$ to both sides of equation~\eqref{eq:con_gjhj_interm2} yields with~$\Tg(\phi) = \partial_\phi \ord{e}{0}(\phi)$ the second equation~\eqref{eq:conj_normal}. 
\end{proof}

\begin{remark} \label{remark:pseudo-inverse}
    Because the matrices $ \Tg(\phi) = \partial_\phi \ord{e}{0}(\phi)$ and $N(\phi)$ are both injective (for each $\phi \in \torus{\torusdim}$), the split $\el$th order homological equations~\eqref{eq:conj_tangent} are equivalent to the equations
    \begin{subequations}
    \begin{align} \label{eq:conj_tangent_ps}
        \partial_{\omega} \ord{g}{\el}(\phi)  + \ord{f}{\el}(\phi) &= \left[\Tg(\phi) \right]^{+} \pi(\phi) \ord{\eta}{\el}(\phi) \\
       \partial_{\omega} \ord{h}{\el}(\phi) - L \ord{h}{\el}(\phi) &= \left[ N(\phi)\right]^{+} (I-\pi(\phi))[\ord{\eta}{\el}(\phi)].\label{eq:conj_normal_ps}
    \end{align}
    \end{subequations}
Here the matrices $\left[\Tg(\phi) \right]^{+}: \mathbb{R}^n \to \mathbb{R}^m$, $\left[N(\phi)\right]^{+}: \mathbb{R}^n \to \mathbb{R}^{n-m}$ are the \emph{pseudo-inverses} of the matrices $\Tg(\phi)$, $N(\phi)$, respectively; i.e., they are the matrices that satisfy
\begin{equation} \label{eq:pseudoinverse}
\begin{aligned}
    \left[\Tg(\phi) \right]^{+} \Tg(\phi) v &= v  & &\text{ and } & \left[\Tg(\phi) \right]^{+} N(\phi) v &= 0  \\
    \left[N(\phi)\right]^{+} N(\phi) w &= w & &\text{ and } & \left[N(\phi)\right]^{+} \Tg(\phi) w &= 0
\end{aligned}
\end{equation}
for all $v \in \mathbb{R}^m$ and $w \in \mathbb{R}^{n-m}$. They can either be constructed using the eigenfunctions of the adjoint variational equation or from the explicit formula $A^{+} = (A^\tr A)^{-1} A^\tr$ for the pseudo-inverse $A^{+}$ of a matrix~$A$. 
\end{remark}

\subsubsection{Solving the ODE}

We next prove that the tangential $\el$th order homological equation~\eqref{eq:conj_tangent} can be solved. However, we point out that the solution is \emph{not} unique, 
essentially because we can parametrize the tangential component of the invariant torus in different ways and different parametrizations  lead to different vector fields. It turns out that we can always choose the parametrization of the invariant torus 
in such a way that the description of the vector field is in \emph{normal form} up to a (fixed) finite order, a concept that we define below: 

\begin{defn}
Denote by $\ord{f}{\el}$ the $\el$th order Taylor coefficient of the vector field~$\ord{f}{\eps}$ (cf.~\eqref{eq:expansion_E_f}) and write its Fourier expansion as 
\[
\ord{f}{\el}(\phi) = \sum_{\textbf{k} \in \mathbb{Z}^n} \ord{\hat{f}}{\el}_{\textbf{k}} e^{i\dotp{\textbf{k}, \phi}}.
\]
Next, fix a finite number $K \in \mathbb{N}$. We say that the function $\ord{f}{\el}$ is in \textbf{normal form up to order $K$}
if 
\[
\ord{\hat{f}}{\el}_{\textbf{k}} = 0
\]
for all $\textbf{k} = (k_1, \ldots, k_m) \in \mathbb{Z}^\torusdim$ with $\dotp{\omega, \textbf{k}} \neq 0$ and $|k| := |k_1| +\dotsb + |k_m| \leq K$, 
i.e., if 
\[
\ord{f}{\el}(\phi) = \sum_{\substack{|k| \leq K \\ \dotp{\omega, \textbf{k}} = 0}} \ord{\hat{f}}{\el}_{\textbf{k}} e^{i\dotp{\textbf{k}, \phi}} + \sum_{|k| > K} \ord{\hat{f}}{\el}_{\textbf{k}} e^{i\dotp{\textbf{k}, \phi}}. 
\]
\end{defn}

\begin{lemma}[Solving the tangential $\el$th order homological equation] \label{lem:SolTang}
    For every smooth function 
    \[
    \ord{\eta}{\el}: \torus{\torusdim} \to \mathbb{R}^\statedim,
    \]
    there exist smooth functions 
    $\ord{g}{\el}, \ord{f}{\el}$ that solve the tangential $\el$th order homological equation~\eqref{eq:conj_tangent}. 
    Moreover, for every finite $K \in \mathbb{N}$, the pair $\ord{g}{\el}, \ord{f}{\el}$ can be chosen in such a way that the function $\ord{f}{\el}$ is in normal form up to order $K$. 
\end{lemma} 
\begin{proof}
Instead of solving~\eqref{eq:conj_tangent} directly, we first apply the pseudo-inverse $\left[\Tg(\phi) \right]^{+}$ to both sides of the equation to obtain the (equivalent) equation~\eqref{eq:conj_tangent_ps} (cf.~Remark~\ref{remark:pseudo-inverse}). This last equation we then explicitly solve as
\begin{align*}
\ord{g}{\el}(\phi) &= 0  & &\text{ and }& \ord{f}{\el}(\phi) &= \left[\Tg(\phi) \right]^{+} \pi(\phi) \ord{\eta}{\el}(\phi),
\end{align*}
which proves the first part of the claim. 

It remains to be shown that the pair $\ord{g}{\el}$, $\ord{f}{\el}$ can be chosen in such a way that $\ord{f}{\el}$ is in normal form up to a finite order $K \in \mathbb{N}$. 
To that end, we introduce the function
\[
\ord{\zeta}{\el}(\phi) = \left[\Tg(\phi) \right]^{+} \pi(\phi) \ord{\eta}{\el}(\phi)
\]
(which corresponds to the the right hand side of~\eqref{eq:conj_tangent_ps}) and write its Fourier expansion as 
\[
\ord{\zeta}{\el}(\phi)  = \sum_{\textbf{k} \in \mathbb{Z}^\torusdim} \ord{\hat{\zeta}}{\el}_{\textbf{k}} e^{i\dotp{\textbf{k}, \phi}}. 
\]
We also write the functions $\ord{f}{\el}, \ord{g}{\el}$ in terms of their Fourier series
\begin{align*}
\ord{f}{\el}(\phi)  &= \sum_{\textbf{k} \in \mathbb{Z}^\torusdim} \ord{\hat{f}}{\el}_{\textbf{k}} e^{i\dotp{\textbf{k}, \phi}}, & \ord{g}{\el}(\phi)  &= \sum_{\textbf{k} \in \mathbb{Z}^\torusdim} \ord{\hat{g}}{\el}_{\textbf{k}} e^{i\dotp{\textbf{k}, \phi}}
\end{align*}
so that the equation~\eqref{eq:conj_tangent_ps} reduces to the equation
\begin{equation} \label{eq:Fourier_tangent}
    i \dotp{\textbf{k}, \omega} \ord{\hat{g}}{\el}_{\textbf{k}} - \ord{\hat{f}}{\el}_{\textbf{k}} = \ord{\hat{\zeta}}{\el}_{\textbf{k}}
\end{equation}
for the Fourier coefficients. For fixed $K \in \mathbb{N}$, the choice 
\begin{align*}
\ord{\hat{g}}{\el}_{\textbf{k}} = 
    \begin{cases}
        \frac{1}{\dotp{\textbf{k} ,\omega}} \ord{\hat{\zeta}}{\el}_{\textbf{k}} & \,\mbox{for } |\textbf{k}| \leq K \mbox{ with } \dotp{\textbf{k}, \omega} \neq 0 \\
        0 &\, \mbox{for } |\textbf{k}| \leq K \mbox{ with } \dotp{\textbf{k}, \omega} = 0 \\
        0 &\, \mbox{for } |\textbf{k}| > K 
    \end{cases}, 
    \qquad \ord{\hat{f}}{\el}_{\textbf{k}} = \begin{cases}
        0 &\, \mbox{for } |\textbf{k}| \leq K \mbox{ with } \dotp{\textbf{k}, \omega} \neq 0 \\
        \ord{\hat{\zeta}}{\el}_{\textbf{k}} &\, \mbox{for } |\textbf{k}| \leq K \mbox{ with } \dotp{\textbf{k}, \omega} = 0 \\
        \ord{\hat{\zeta}}{\el}_{\textbf{k}} &\, \mbox{for } |\textbf{k}| > K 
    \end{cases}
\end{align*}
leads to two convergent Fourier series that solve the equation~\eqref{eq:conj_tangent_ps} and for which $\ord{f}{\el}$ is in normal form. This proves the second part of the claim. 
\end{proof}

We next solve the normal $\el$th order homological equation~\eqref{eq:conj_normal}. We point out that this equation \emph{does} have a unique solution (in contrast to the tangential equation~\eqref{eq:conj_tangent}): 
Ww have freedom in parameterizing the perturbed torus in the tangential direction, but once we have done so, the parametrization in the normal direction is fixed. 

\begin{lemma}[Solving the normal $\el$th order homological equation]\label{lem:SolNorm}
    For every smooth function 
    \[
    \ord{\eta}{\el}: \torus{m} \to \mathbb{R}^\statedim,
    \]
    there exist a smooth function
    $\ord{h}{\el}$ that solves the normal $\el$th order homological equation~\eqref{eq:conj_normal}.  
\end{lemma}
\begin{proof}
Instead of solving~\eqref{eq:conj_normal} directly, we first apply the pseudo-inverse $\left[N(\phi)\right]^{+}$ to both sides of the equation to obtain the (equivalent) equation~\eqref{eq:conj_normal_ps} (cf.~Remark~\ref{remark:pseudo-inverse}).
We next introduce the function
\[
\ord{K}{\el}(\phi) = \left[ N(\phi)\right]^{+} (I-\pi(\phi)[\ord{\eta}{\el}(\phi)]
\]
(which corresponds to the right hand side of~\eqref{eq:conj_normal_ps}) and write the functions~$\ord{h}{\el}$ and~$\ord{K}{\el}$ in terms of their Fourier series: 
\begin{align*}
    \ord{K}{\el}(\phi)  &= \sum_{\textbf{k} \in \mathbb{Z}^\torusdim} \ord{\hat{K}}{\el}_{\textbf{k}} e^{i\dotp{\textbf{k}, \phi}}, & 
    \ord{h}{\el}(\phi)  &= \sum_{\textbf{k} \in \mathbb{Z}^\torusdim} \ord{\hat{h}}{\el}_{\textbf{k}} e^{i\dotp{\textbf{k}, \phi}}.
\end{align*}
From here, the equation~\eqref{eq:conj_normal_ps} reduces to the equations 
\begin{equation} \label{eq:Fourier_normal}
    i \dotp{\textbf{k}, \omega} \ord{\hat{h}}{\el}_\textbf{k} - L \ord{\hat{h}}{\el}_\textbf{k} = \ord{\hat{K}}{\el}_\textbf{k}
\end{equation}
for the Fourier coefficients. 
Since the matrix~$L$ is hyperbolic and thus has no spectrum on the imaginary axis (cf.~Lemma~\ref{lem:N_L}), 
the matrix $i \dotp{\textbf{k}, \omega} I - L$ is invertible for all $\textbf{k} \in \mathbb{Z}^\torusdim$, and the sequence 
\[
\left \lbrace \left \lVert \left(i \dotp{\textbf{k}, \omega} I - L\right)^{-1} \right \rVert \right \rbrace_{\textbf{k} \in \mathbb{Z}^\torusdim} 
\]
is bounded. Therefore each equation in~\eqref{eq:Fourier_normal} has a unique solution 
\[
\ord{\hat{h}}{\el}_\textbf{k} = \left(i \dotp{\textbf{k}, \omega} I - L\right)^{-1}\ord{\hat{K}}{\el}_\textbf{k}
\]
and this choice of coefficients leads to a convergent Fourier series for $\ord{h}{\el}$ that solves~\eqref{eq:conj_normal_ps}. 
\end{proof}

\subsubsection{Solving the Transport Equation}

Having solved the tangential $\el$th order equation~\eqref{eq:conj_tangent} for~$\ord{g}{\el}$ and~$\ord{f}{\el}$, and having solved the normal $\el$th order equation for~$\ord{h}{\el}$, we use again the coordinate transformation $\ord{e}{\el}(\phi) = \Tg(\phi) \ord{g}{\el}(\phi) + N(\phi) \ord{h}{\el}(\phi)$ to find a solution pair $\ord{e}{\el}, \ord{f}{\el}$ for the $\el$th order homological equation~\eqref{eq:con_order_j_n1} in the original coordinate frame (cf.~Lemma~\ref{lem:con_order_splitted}). The only task now left is to solve equation~\eqref{eq:con_order_j_n2}, which corresponds to the transport part of the $\el$th order homological equations. This is

\begin{lemma}\label{lem:TranspEq}
Suppose that the functions $\ord{e}{\el}, \ord{f}{\el}$ solve the finite dimensional $\el$th order homological equation~\eqref{eq:con_order_j_n1}, then the function
\begin{equation} \label{eq:Ej_phi_s}
    \begin{aligned}
\ord{E}{\el}(\phi, s) &= \ord{e}{\el}(\phi + s \omega) + \partial_\varphi \ord{E}{0}(\phi, s) \cdot \int_0^s \ord{f}{\el}(\phi+\omega(s-\zeta))d \zeta \\ &\qquad- \int_0^s \ord{H}{\el}(\phi + \omega(s-\zeta), \zeta) d \zeta    
\end{aligned}
\end{equation}
solves transport $\el$th order homological equation~\eqref{eq:con_order_j_n2} together with the boundary condition $\ord{E}{\el}(\phi, 0) = \ord{e}{\el}(\phi)$.
\end{lemma}
\begin{proof}
We divide the proof into two steps: 

\medskip
\noindent \textsc{Step~1:}
We first fix two $C^1$ functions
\begin{align*}
u: \torus{\torusdim} &\to \mathbb{R}^\statedim  & &\text{ and } & V: \torus{\torusdim} \times [-\tau, 0] &\to \mathbb{R}^\statedim
\end{align*}
and show that the partial differential equation
\begin{align} \label{eq:PDE_U}
    \begin{cases}
        \partial_\phi U(\phi, s) \cdot \omega - \partial_s U(\phi, s) = V(\phi, s), &\quad \phi \in \torus{\torusdim}, \, s \in [-\tau, 0] \\
        U(\phi, 0) = u(\phi), &\quad \phi \in \torus{\torusdim}
    \end{cases}
\end{align}
is solved by the function 
\begin{equation} \label{eq:U}
    U(\phi, s) = u(\phi + s \omega) - \int_0^s V(\phi + (s-\zeta)\omega, \zeta) d \zeta. 
\end{equation} 
It follows directly from the definition of $U$ that $U(\phi, 0) = u(\phi)$. Moreover, computing $\partial_\phi U(\phi, s)\cdot \omega$ from the definition of~$U$ in~\eqref{eq:U} gives that
\[
\partial_\phi U(\phi, s)\cdot \omega = u'(\varphi + s \omega) \cdot \omega - \int_0^s V'(\phi + (s-\zeta)\omega, \zeta) \cdot \omega d \zeta 
\]
and simultaneously computing $\partial_s U(\phi, s)$ yields 
\[
\partial_s U(\phi, s) = u'(\phi + s \omega) \cdot \omega - \int_0^s V'(\phi + (s-\zeta)\omega, \zeta) \cdot \omega d \zeta - V(\phi, \zeta).
\]
Hence $\partial_\phi U(\phi, s)\cdot \omega - \partial_s U(\phi, s) = V(\phi, \zeta)$, which proves that the function $U$ defined in~\eqref{eq:U} solves the partial differential equation~\eqref{eq:PDE_U}. 

\medskip
\noindent \textsc{Step~2:} Applying Step~1 with 
\begin{align*}
    u(\phi) &= \ord{e}{\el}(\phi) \\
    V(\phi, s) &= - \partial_\phi \ord{E}{0}(\phi, s) \cdot \ord{f}{\el}(\phi) + \ord{H}{\el}(\phi, s)
\end{align*}
gives that the function 
\begin{equation*} 
\begin{aligned}
    \ord{E}{\el}(\varphi, s) &= \ord{e}{\el}(\phi + s \omega) + \int_0^s \partial_\phi \ord{E}{0}(\phi + (s-\zeta) \omega, \zeta) \cdot \ord{f}{\el}(\phi + (s-\zeta)\omega) d \zeta \\
    &\qquad - \int_0^s \ord{H}{\el}(\phi + (s-\zeta) \omega, \zeta) d \zeta
\end{aligned}
\end{equation*}
solves the partial differential equation~\eqref{eq:con_order_j_n2} together with the boundary condition $\ord{E}{\el}(\phi, 0) = \ord{e}{\el}(\phi)$. From the definition of $\ord{E}{0}$ in~\eqref{eq:E0f0} we compute 
\begin{align*}
    \ord{E}{0}_j(\phi + (s-\zeta) \omega, \zeta) &= \Gamma_j \left( \frac{\phi_j + (s-\zeta) \omega_j}{\omega_j} + \zeta\right) \\
    &= \Gamma_j \left( \frac{\phi_j}{\omega_j} + s \right) =
     \ord{E}{0}_j(\phi, s)
\end{align*}
for all~$1\leq j\leq \torusdim$.
This implies that $\ord{E}{0}(\phi + (s-\zeta) \omega, \zeta) =  \ord{E}{0}(\phi, s)$ and hence also $\partial_\phi \ord{E}{0}(\phi + (s-\zeta) \omega, \zeta) = \partial_\varphi \ord{E}{0}(\phi, s)$ for all $(\phi, s) \in \torus{\torusdim} \times [-\tau, 0]$. From there the expression for $\ord{E}{\el}$ reduces to
\[
\begin{aligned}
 \ord{E}{\el}(\varphi, s) &= \ord{e}{\el}(\phi + s \omega) + \partial_\phi \ord{E}{0}(\phi, s) \cdot \int_0^s \ord{f}{\el}(\phi + (s-\zeta)\omega) d \zeta \\
    &\qquad- \int_0^s \ord{H}{\el}(\phi + (s-\zeta) \omega, \zeta) d \zeta, 
\end{aligned}
\]
which proves the lemma. 
\end{proof}

\subsection{Summary of the algorithm} \label{subsec:algorithm}

The results of this section yield an algorithm to compute the expansions of the phase-reduced vector field~$\ord{f}{\eps}$ and of the torus embeddings~$\ord{e}{\eps}, \ord{E}{\eps}$ given by
\begin{align*}
    \ord{f}{\eps} &= \ord{f}{0} + \varepsilon \ord{f}{1} +\dotsb + \eps^\el \ord{f}{\el} +\dotsb \\
    \ord{e}{\eps} & = \ord{e}{0} + \varepsilon \ord{e}{1} +\dotsb + \eps^\el \ord{e}{\el} +\dotsb \\
    \ord{E}{\eps} &= \ord{E}{0} + \varepsilon \ord{E}{1} +\dotsb + \eps^\el \ord{E}{\el} +\dotsb    
\end{align*}
up to arbitrary order in~$\epsilon$: 

\medskip
\noindent
\framebox{\textbf{Starting point $\el = 0$}} 

\medskip
\noindent
Compute the functions $\ord{e}{0}$, $\ord{E}{0}$ and $\ord{f}{0}$ as 
\begin{align*}
    \ord{e}{0}(\phi) & = \left( \Gamma_1 \left(\frac{\varphi_1}{\omega_1}\right), \ldots, \Gamma_n \left(\frac{\varphi_\torusdim}{\omega_\torusdim} \right) \right) \\
     \ord{E}{0}(\varphi, s) &= \left( \Gamma_1 \left(\frac{\varphi_1}{\omega_1}+s \right), \ldots, \Gamma_\torusdim \left(\frac{\varphi_\torusdim}{\omega_\torusdim}+s \right) \right) \\
     \ord{f}{0}(\varphi) &\equiv \left(\omega_1, \ldots, \omega_\torusdim\right)
\end{align*}
where~$\Gamma_j$ for $1 \leq j \leq \torusdim$ is the periodic orbit of the uncoupled ODE $\dot{x}_j = F_j(x_j)$ and~$\omega_j$ is its frequency (cf.~Hypothesis~\ref{hyp:system}).

\medskip
\noindent
\framebox{\textbf{Iteration $\el-1 \to \el$}} 

\medskip
\noindent
Having computed $\ord{f}{\el'}$, $\ord{e}{\el'}$ and $\ord{E}{\el'}$ for all $1 \leq \el' \leq \el-1$, compute $\ord{f}{\el}$, $\ord{e}{\el}$ and $\ord{E}{\el}$ as follows: 
\begin{itemize}
    \item [(i)] Compute $\ord{\eta}{\el}$ and $\ord{H}{\el}$, which corresponds to the contribution of the lower order terms to the $\el$th order homological equations~\eqref{eq:con_order_j_n1}--\eqref{eq:con_order_j_n2}, respectively. 
    For the orders $j = 1, 2$, the explicit formulas for~$\ord{\eta}{\el}$ are 
    \begin{align*}
        \ord{\eta}{1}(\phi) &= G(\ord{E}{0}(\phi, \, . )) \\
        \ord{\eta}{2}(\phi) &= \frac{1}{2} F''(\ord{e}{0}(\phi))[\ord{e}{1}(\phi), \ord{e}{1}(\phi)] -  \partial_\phi \ord{e}{1}(\phi) \cdot \ord{f}{1}(\phi) \\ 
        & \qquad + G'(\ord{E}{0}(\phi, \, .))\ord{E}{1}(\phi, \, .) ,
    \intertext{where~$F$ is the vector field in the uncoupled system and~$G$ is the coupling functional in~\eqref{eq:DDE}; 
    and the explicit formulas for~$\ord{H}{\el}$ are}
        \ord{H}{1}(\phi, s) & \equiv 0 \\
        \ord{H}{2}(\phi, s) &= - \partial_{\phi} \ord{E}{1}(\phi, s) \cdot \ord{f}{1}(\phi). 
    \end{align*}
    \item[(ii)] Solve the tangential $\el$th order homological equation 
    \[
    \Tg(\phi) [\partial_{\omega} \ord{g}{\el}(\phi)  + \ord{f}{\el}(\phi)] = \pi(\phi) \ord{\eta}{\el}(\phi)
    \]
    for $\ord{g}{\el}$ and $\ord{f}{\el}$. This computation can be done in Fourier basis as follows: Denote by $\left[ \Tg(\phi) \right]^+$ the pseudo-inverse of the matrix $\Tg(\phi) = \partial_\phi \ord{e}{0}(\phi)$; cf.~Equation~\eqref{eq:pseudoinverse}. Denote by $\ord{\hat{\zeta}}{\el}_{\textbf{k}}$ the Fourier coefficients of the function $\phi \mapsto \left[\Tg(\phi) \right]^{+} \pi(\phi) \ord{\eta}{\el}(\phi)$, that is,
    \[
    \left[\Tg(\phi) \right]^{+} \pi(\phi) \ord{\eta}{\el}(\phi) = \sum_{\textbf{k} \in \mathbb{Z}^\torusdim} \ord{\hat{\zeta}}{\el}_{\textbf{k}} e^{i\dotp{\textbf{k}, \phi}}. 
    \]
    Then the Fourier coefficients $\ord{\hat{g}}{\el}_{\textbf{k}}$ of $\ord{g}{\el}$ and the Fourier coefficients $\ord{\hat{f}}{\el}_{\textbf{k}}$ of $\ord{f}{\el}$ 
    \begin{align*}
    \ord{g}{\el}(\phi)  &= \sum_{\textbf{k} \in \mathbb{Z}^\torusdim} \ord{\hat{g}}{\el}_{\textbf{k}} e^{i\dotp{\textbf{k}, \phi}}, & \ord{f}{\el}(\phi)  &= \sum_{\textbf{k} \in \mathbb{Z}^\torusdim} \ord{\hat{f}}{\el}_{\textbf{k}} e^{i\dotp{\textbf{k}, \phi}}
    \end{align*}
    satisfy the equations 
    \[
    i \dotp{\textbf{k}, \omega} \ord{\hat{g}}{\el}_{\textbf{k}} - \ord{\hat{f}}{\el}_{\textbf{k}} = \ord{\hat{\zeta}}{\el}_{\textbf{k}}. 
    \]
    In particular, the vector field $\ord{f}{\el}$ can be chosen in normal form up to order $K \in \mathbb{N}$ by setting 
    \begin{align*} \label{eq:nf_sol}
\ord{\hat{g}}{\el}_{\textbf{k}} &= 
    \begin{cases}
        \frac{1}{\dotp{\textbf{k} ,\omega}} \ord{\hat{\zeta}}{\el}_{\textbf{k}} & \,\mbox{for } |\textbf{k}| \leq K \mbox{ with } \dotp{\textbf{k}, \omega} \neq 0 \\
        0 &\, \mbox{for } |\textbf{k}| \leq K \mbox{ with } \dotp{\textbf{k}, \omega} = 0 \\
        0 &\, \mbox{for } |\textbf{k}| > K 
    \end{cases}, \\
     \ord{\hat{f}}{\el}_{\textbf{k}} &= \begin{cases}
        0 &\, \mbox{for } |\textbf{k}| \leq K \mbox{ with } \dotp{\textbf{k}, \omega} \neq 0 \\
        \ord{\hat{\zeta}}{\el}_{\textbf{k}} &\, \mbox{for } |\textbf{k}| \leq K \mbox{ with } \dotp{\textbf{k}, \omega} = 0 \\
        \ord{\hat{\zeta}}{\el}_{\textbf{k}} &\, \mbox{for } |\textbf{k}| > K. 
    \end{cases}
\end{align*}
    \item[(iii)] Solve the normal $\el$th order homological equation 
    \[
    N(\phi)[\partial_{\omega} \ord{h}{\el}(\phi) - L \ord{h}{\el}(\phi)] = (I-\pi(\phi))[\ord{\eta}{\el}(\phi)]
    \]
    for $\ord{h}{\el}$. This computation can be done in Fourier basis as follows: 
    Denote by $\left[ N(\phi) \right]^{+}$ pseudo-inverse of the matrix~$N(\phi)$ (see Lemma~\ref{lem:N_L} for the definition of~$N(\phi)$ and equation~\eqref{eq:pseudoinverse} for the definition of its pseudo-inverse~$\left[ N(\phi) \right]^{+}$). 
    Denote by $\ord{\hat{K}}{\el}_{\textbf{k}}$ the Fourier coefficients of the function $\phi \mapsto \left[ N(\phi)\right]^{+} (I-\pi(\phi))[\ord{\eta}{\el}(\phi)]$, i.e., 
    \[
    \left[ N(\phi)\right]^{+} (I-\pi(\phi))[\ord{\eta}{\el}(\phi)]  = \sum_{\textbf{k} \in \mathbb{Z}^\torusdim} \ord{\hat{K}}{\el}_{\textbf{k}} e^{i\dotp{\textbf{k}, \phi}}.
    \]
    Then the Fourier coefficients $\ord{\hat{h}}{\el}_{\textbf{k}}$ of the function $\ord{h}{\el}$
    \[
    \ord{h}{\el}(\phi)  = \sum_{\textbf{k} \in \mathbb{Z}^\torusdim} \ord{\hat{h}}{\el}_{\textbf{k}} e^{i\dotp{\textbf{k}, \phi}}
    \]
    are given by 
    \[
    \ord{\hat{h}}{\el}_\textbf{k} = \left(i \dotp{\textbf{k}, \omega} I - L\right)^{-1}\ord{\hat{K}}{\el}_\textbf{k},
    \]
    where the matrix $L$, which is related to the hyperbolic Floquet exponents of the uncoupled problem, is defined in Lemma~\ref{lem:N_L}. 
    \item[(iv)] Compute $\ord{e}{\el}$, which corresponds to the $\el$th order Taylor coefficient of the torus embedding $\ord{e}{\eps}$, as 
    \[
    \ord{e}{\el}(\phi) = \Tg(\phi) \ord{g}{\el}(\phi) + N(\phi) \ord{h}{\el}(\phi).
    \]
    \item[(v)] Compute $\ord{E}{\el}$, which corresponds to the $\el$th order Taylor coefficient of the torus embedding $\ord{E}{\eps}$, as 
    \begin{align*}
    \ord{E}{\el}(\phi, s) &= \ord{e}{\el}(\phi + s \omega) + \partial_\varphi \ord{E}{0}(\phi, s) \cdot \int_0^s \ord{f}{\el}(\phi+\omega(s-\zeta))d \zeta \\ &\qquad- \int_0^s \ord{H}{\el}(\phi + (s-\zeta)\omega, \zeta) d \zeta   .
    \end{align*}
\end{itemize}

\section{Computing phase reduction for two coupled Stuart--Landau oscillators}
\label{sec:sl}

\newcommand{\C}{\mathbb{C}}

In this section, we use the approach developed in Section~\ref{sec:method} to compute a second order phase reduction of a concrete model. 
Specifically, we consider the dynamics of two identical Stuart--Landau oscillators with time-delayed diffusive coupling whose states $z_1,z_2\in \C$ evolve according to
\begin{equation} \label{eq:sl_coupled}
\begin{aligned}
\dot{z}_1 &= (\alpha + i \beta) z_1 + (\gamma +i \delta) \left| z_1 \right|^2 z_1 + \epsilon e^{i \cc} \left(z_2(t-\tau) - z_1 \right) \\
\dot{z}_2 &= (\alpha + i \beta) z_2 + (\gamma +i \delta) \left| z_2 \right|^2 z_2 + \epsilon e^{i \cc} \left(z_1(t-\tau) - z_2 \right). 
\end{aligned}
\end{equation}
Here the parameters $\alpha, \beta, \gamma, \delta \in \mathbb{R}$ are the parameters of the Stuart--Landau oscillators, $\epsilon\geq 0$ is a small coupling constant, $\cc \in [0, 2 \pi]$ describes the rotation in the coupling and $\tau >0$ is the time delay. We write the coupled Stuart--Landau system~\eqref{eq:sl_coupled} in the form of the general DDE~\eqref{eq:DDE} by defining
\begin{subequations}
\begin{align} \label{eq:F_sl}
    F(z) &= \begin{pmatrix}
        \comp{F}{1}(z_1) \\
        \comp{F}{2}(z_2)
    \end{pmatrix} 
    =
    \begin{pmatrix}
        (\alpha + i \beta) z_1 + (\gamma + i \delta) \left| z_1 \right|^2 z_1 \\
        (\alpha + i \beta) z_2 + (\gamma + i \delta) \left| z_2 \right|^2 z_2
    \end{pmatrix} \\  \label{eq:G_sl}
    G(Z(t, \cdot)) &= G \begin{pmatrix}
        Z_1(t, \, \cdot) \\ Z_2(t, \, \cdot)
    \end{pmatrix}
    = 
    \begin{pmatrix}
        e^{i \cc}[z_2(t-\tau) - z_1(t)] \\
        e^{i \cc}[z_1(t-\tau) - z_2(t)]
    \end{pmatrix}, 
\end{align}
\end{subequations}
where we define the history segments $Z_j(t, \, \cdot)$ via $Z_j(t, s) = z_j(t+s)$ for $s \in [-\tau, 0]$ and $j = 1, 2$ and we additionally have made the identification $\C=\R^2$. 

For $\gamma \neq 0$, each of the (identical) ODEs
\begin{equation} \label{eq:sl_single}
    \dot{z}_j(t) = \ord{F}{\el}(z_j(t)) =  (\alpha + i \beta) z_j(t) + (\gamma + i \delta) \left| z_j(t) \right|^2 z_j(t), \qquad j \in \{1, 2\}
\end{equation}
has a periodic orbit 
\begin{subequations}
\begin{equation} \label{eq:po_exp}
    \Gamma_j(t) = Re^{i \Omega t}
\end{equation}
with radius $R$ and angular frequency $\Omega$ given by 
\begin{equation} \label{eq:R_omega_sl}
R = \sqrt{-\frac{\alpha}{\gamma}}, \qquad \Omega = \beta - \frac{\alpha \delta}{\gamma}. 
\end{equation}
\end{subequations}
The product of the two periodic orbits forms a two-dimensional invariant torus for the uncoupled system of Stuart--Landau oscillators, and this torus persists for small nonzero values of the coupling parameter $\epsilon$. 

The main result of this section is the computation phase-reduced dynamics to first and second order, as summarized in the following two lemmas: 

\begin{lemma}[First order phase reduction] \label{lem:phasereductiontwoSLoscillators}

Consider the delay-coupled Stuart--Landau oscillators~\eqref{eq:sl_coupled} 
with periodic orbits in the uncoupled system given by ~\eqref{eq:po_exp}--\eqref{eq:R_omega_sl}. 
The first order phase reduced dynamics for $\psi := \phi_1 - \phi_2$ is given by
\begin{align}\label{eq:SLfo}
\dot{\psi} = -2 \epsilon \left[\cos(\cc- \omega \tau) +\frac{\delta}{\gamma} \sin(\cc - \omega \tau) \right] \sin(\psi).    
\end{align}
\end{lemma}

The first-order phase dynamics~\eqref{eq:SLfo} are governed by a single harmonic. This restricts the possible dynamics that is predicted by this approximation.
In-phase synchrony ($\psi=0$) and anti-phase synchrony ($\psi=\pi$) are the only isolated equilibria of~\eqref{eq:SLfo}; we point out that the existence of these equilibria is also expected from the symmetry properties of the coupled Stuart--Landau system~\eqref{eq:sl_coupled}. These two equilibria can exchange stability in a degenerate bifurcation; 
in particular, bistability between in-phase and anti-phase synchrony is not possible in the first-order phase approximation.
To understand the dynamics of the coupled Stuart--Landau system~\eqref{eq:sl_coupled} for parameter values near this bifurcation point, one has to take second-order corrections into account. 

\begin{lemma}[Second order phase reduction for $\delta = 0$]\label{lem:SLso}
Consider the delay-coupled Stuart--Landau oscillators~\eqref{eq:sl_coupled} 
with periodic orbits in the uncoupled system given by ~\eqref{eq:po_exp}--\eqref{eq:R_omega_sl}. 
The second order phase reduced dynamics for $\psi := \phi_1 - \phi_2$ is given by
\begin{equation}\label{eq:SLso}
\begin{aligned}
    \dot{\psi} 
    = &-2 \epsilon \cos(\cc- \omega \tau) \sin(\psi)  \\
    &\qquad+ \epsilon^2 \left( 2 \tau \sin(\cc) \sin(\cc - \omega \tau) \sin(\psi) - \left[\tau + \frac{1}{ \alpha} \sin^2(\cc -  \omega  \tau) \right]\sin(2\psi) \right).
\end{aligned}
\end{equation}
\end{lemma}

The time delay enters the dynamical equations explicitly in the second order corrections to the phase dynamics.
These corrections also introduce a second harmonic that affects the dynamics:
the second harmonic allows for phase-locked equilibria to arise beyond the in-phase and anti-phase synchrony; cf.~\cite{Rusin2010}.
If these nontrivial phase-locked equilibria are unstable, this can yield bistability between in-phase and anti-phase synchrony in~\eqref{eq:SLso}.


We will return to the second order phase dynamics~\eqref{eq:SLso} and its relation to the delay-coupled SL oscillators in Section~\ref{sec:bistab}, and present the computations that lead to Lemmas~\ref{lem:phasereductiontwoSLoscillators} and~\ref{lem:SLso} in the remainder of this section.
For the first-order approximation (Lemma~\ref{lem:phasereductiontwoSLoscillators}), we solve the $\el$-order homological equations~\eqref{eq:conj_tangent}--\eqref{eq:conj_normal} for $\el=0$ and $\el=1$ to obtain the zero-th order information $\ord{f}{0}$, $\ord{e}{0}$ and $\ord{E}{0}$ together with the
first-order phase dynamics~$\ord{f}{1}$ (and $\ord{e}{1}$ through its tangential and normal components $\ord{g}{1}, \ord{h}{1}$).
For the second-order approximation (Lemma~\ref{lem:SLso}), we solve the $\el$-order homological equations~\eqref{eq:conj_order} for $\el=2$ to compute the second order corrections~$\ord{f}{2}$.
This requires knowledge of the first-order embeddings~$\ord{e}{1}, \ord{E}{1}$ of state and history.

\subsection{Zeroth and first order phase reduction}

Lemma~\ref{lem:E0f0} yields that the zero-th order homological equations~\eqref{eq:con_order_0} for the dynamics on the torus are solved by
\begin{equation}
    \label{eq:e0_sl}
    \begin{aligned}
    \ord{f}{0}(\phi) &= (\Omega, \Omega)^\tr \\
\ord{e}{0}(\phi) &= (Re^{i \phi_1}, Re^{i \phi_2})^\tr \\ \ord{E}{0}(\phi, s) &= (Re^{i (\phi_1 + \Omega s)}, Re^{i (\phi_2 + \Omega s)})^\tr
\end{aligned}
\end{equation}
As in Section~\ref{sec:iterative_con}, we will denote by 
\begin{equation*}
     \omega := \ord{f}{0}(\phi)= (\Omega, \Omega)
\end{equation*}
the vector consisting of the frequencies of the periodic orbits. 

\medskip

To solve the first order homological equations~\eqref{eq:conj_tangent}--\eqref{eq:conj_normal} with $\el = 1$, we first find an explicit expression for the matrix--valued function $\phi \mapsto N(\phi)$
whose columns span the normal bundle of the unperturbed torus (cf.~Lemma~\ref{lem:N_L}): 

\begin{lemma}
Let $F: \mathbb{R}^n \to \mathbb{R}^n$ be the vector field of the uncoupled Stuart--Landau oscillators defined in~\eqref{eq:F_sl}, let $\omega = (\Omega, \Omega)$ be the vector consisting of the frequencies of the periodic orbits in the uncoupled system and let $\ord{e}{0}$ defined in~\eqref{eq:e0_sl}
be the embedding of the unperturbed invariant torus in the uncoupled system. 
Moreover, let
\begin{equation} \label{eq:Tg_SL}
    T(\phi) := \partial_\phi \ord{e}{0}(\phi) = \begin{pmatrix}
    i R e^{i \phi_1} & 0 \\
    0 & i R e^{i \phi_2}
\end{pmatrix}
\end{equation}
be the matrix with the property that, for each $\phi \in \torus{2}$, the columns of $T(\phi)$ form a basis for the tangents space to the unperturbed torus at the point $\ord{e}{0}(\phi)$.

Then the family of smooth maps 
\[ N: \torus{2} \to \mathcal{L}(\mathbb{R}^2, \mathbb{C}^2)\]
defined by 
\begin{equation} \label{eq:N_sl}
    N (\phi) \begin{pmatrix}
        u_1 \\ u_2
    \end{pmatrix}
= \begin{pmatrix}
    (\gamma + i \delta) e^{i \phi_1} u_1 \\
    (\gamma + i \delta) e^{i \phi_2} u_2
\end{pmatrix}
\end{equation}
and the hyperbolic matrix 
\[
L \in \mathbb{R}^{2 \times 2}
\]
given by 
\begin{equation*}
    L \begin{pmatrix}
        u_1 \\ u_2
    \end{pmatrix}
     = \begin{pmatrix}
         -2 \alpha u_1 \\
         - 2 \alpha u_2
     \end{pmatrix}
\end{equation*}
have the following properties: 
\begin{enumerate}
\item For each $\phi \in \torus{2}$, the image of the map $N(\phi)$ is transverse to image of the map $T(\phi)$; this means that  
    \begin{equation*} 
    \im \Tg(\phi)   \oplus \im N(\varphi)= \mathbb{C}^{2}; 
    \end{equation*}
    \item The family $N(\phi)$ satisfies the differential equation
  \begin{equation} 
        \label{eq:pde_N_sl}\partial_{\omega} N(\phi) + N(\phi) L = F'(\ord{e}{0}(\phi)) N(\phi).
  \end{equation}
\end{enumerate}
\end{lemma}
\begin{proof} 
We verify that the images of~$\Tg(\phi)$ and~$N(\phi)$ are transverse. To do so, we fix two real-valued vectors $(u_1, u_2) \in \mathbb{R}^2, \, (v_1, v_2) \in \mathbb{R}^2$ and fix a complex-valued vector $(w_1, w_2) \in \mathbb{C}^2$; these vectors satisfy the relation 
\[
\Tg(\phi) \begin{pmatrix}
    u_1 \\ u_2
\end{pmatrix} 
+ N(\phi)
\begin{pmatrix}
    v_1 \\ v_2
\end{pmatrix} = 
\begin{pmatrix}
    w_1 \\ w_2
\end{pmatrix}
\]
if and only if the identities
\begin{equation} \label{eq:transverse_sl}
    e^{i \phi_j}(i R u_j + (\gamma + i \delta) v_j) = w_j, \qquad j = 1, 2
\end{equation}
hold. But since $\gamma \neq 0$, the identities~\eqref{eq:transverse_sl} have a unique solution 
\[
v_j = \frac{1}{\gamma} \cos(\phi_j) w_j, \qquad u_j = \frac{1}{R} \left( - \sin(\phi_j) w_j - \frac{\delta}{\gamma} \cos(\phi_j) w_j \right)
\]
which  proves that the images of~$\Tg(\phi)$ and~$N(\phi)$ are transverse. 

   \medskip

We next prove that the map $\phi \mapsto N(\phi)$ solves the differential equation~\eqref{eq:pde_N_sl}. 
We fix a point $(z_1, z_2) \in \mathbb{C}^2$ and a vector $(v_1, v_2) \in \mathbb{C}^2$ and calculate that the directional derivative $F'(z_1, z_2)(v_1, v_2)$ is given by 
\begin{equation} \label{eq:F_der}
F'\begin{pmatrix}
    z_1 \\ z_2
\end{pmatrix}
\begin{pmatrix}
    v_1 \\ v_2
\end{pmatrix}
= 
\begin{pmatrix}
    (\alpha + i \beta) v_1 + (\gamma + i \delta) (2 |z_1|^2 v_1 + z_1^2 \bar{v}_1) \\
    (\alpha + i \beta) v_2 + (\gamma + i \delta) (2 |z_2|^2 v_2 + z_2^2 \bar{v}_2)
\end{pmatrix}.
\end{equation}
Given another (real-valued) vector $(u_1, u_2) \in \mathbb{R}^2$, this implies that
\begin{equation}
\begin{aligned}
F'(\ord{e}{0}(\phi)) N(\phi) \begin{pmatrix}
    u_1 \\ u_2
\end{pmatrix}
&= \left[ \alpha + i \beta + R^2 (3 \gamma + i \delta) \right]
\begin{pmatrix}
(\gamma + i \delta) e^{i \phi_1} \\
 (\gamma + i \delta) e^{i \phi_2}
\end{pmatrix} \\
&= [ i \beta + R^2 i \delta- 2 \alpha ] N(\phi) \begin{pmatrix}
    u_1 \\ u_2
\end{pmatrix},
\end{aligned}
\end{equation}
where in the last step we used the expression for the radius~$R$ given in~\eqref{eq:R_omega_sl} and the expression for~$N$ in~\eqref{eq:N_sl}. 
Simultaneously, we compute from the expressions for $N(\phi)$ and $L$ that
\begin{align*}
    \partial_{\omega} N(\phi) \begin{pmatrix}
        u_1 \\ u_2
    \end{pmatrix} + N(\phi) L \begin{pmatrix}
        u_1 \\ u_2
    \end{pmatrix}
& = (i \Omega - 2 \alpha) \begin{pmatrix}
    (\gamma + i \delta) e^{i \phi_1} u_1 \\
    (\gamma + i \delta) e^{i \phi_2} u_2
\end{pmatrix} \\
& = (i \beta + i \delta R^2 - 2 \alpha) N(\phi) \begin{pmatrix}
        u_1 \\ u_2
    \end{pmatrix},
\end{align*}
where in the last step we used the expression for the frequency~$\Omega$ in~\eqref{eq:R_omega_sl}. 
Comparing the above expression for $\partial_{\omega} N(\phi) + N(\phi) L$ with the expression for $F'(\ord{e}{0}(\phi)) N(\phi)$, it follows that the map $\phi \mapsto N(\phi)$ solves the differential equation~\eqref{eq:pde_N_sl}. 
\end{proof}

Having constructed the map $\phi \mapsto N(\phi)$, we next compute that the family of projection operators
\[
\phi \mapsto N(\phi) \in \mathcal{L}(\mathbb{C}^2, \mathbb{C}^2)
\]
with image equal to $\im(\Tg(\phi))$ and kernel equal to $\im(N(\phi))$ are given by 
\begin{equation}
\label{eq:sl_pi}
\begin{aligned} 
    \pi(0) \begin{pmatrix}
    z_1 \\ z_2
\end{pmatrix} &= \begin{pmatrix}
    i (- \frac{\delta}{\gamma} \Re(z_1) + \Im(z_1)) \\
    i (- \frac{\delta}{\gamma} \Re(z_2) + \Im(z_2))
\end{pmatrix} \\ 
\pi(\phi)&= \begin{pmatrix}
    e^{i \phi_1} & 0 \\
    0 & e^{i \phi_2}
\end{pmatrix}
\pi(0) 
\begin{pmatrix}
    e^{-i \phi_1} & 0 \\
    0 & e^{-i \phi_2}
\end{pmatrix}, 
\end{aligned}
\end{equation}
By first computing $\ord{\eta}{1}$ from the expression in~\eqref{eq:G1G2_expl},
we find that the right hand side of the first order tangential homological equation~\eqref{eq:conj_tangent} is given by
\begin{align*}
    \pi(\phi) \ord{\eta}{1}(\phi) &= iR\begin{pmatrix}
        -\frac{\delta}{\gamma} e^{i \phi_1} \left[\cos(\phi_2 - \phi_1 - \Omega \tau + \cc) - \cos(\cc)\right]  \\
        -\frac{\delta}{\gamma} e^{i \phi_2} \left[\cos(\phi_1 - \phi_2 - \Omega \tau + \cc) - \cos(\cc)\right]
    \end{pmatrix} \\
    &\qquad +
    i R\begin{pmatrix}
       e^{i \phi_1} \left[ \sin(\phi_2 - \phi_1 - \Omega \tau + \cc) - \sin(\cc) \right]\\
       e^{i \phi_2} \left[ \sin(\phi_1 - \phi_2 - \Omega \tau + \cc) - \sin(\cc) \right]
    \end{pmatrix}.
\end{align*}
Together with $\Tg(\phi) = \mbox{diag}(iR e^{i \phi_1}, iR e^{i \phi_2})$ as defined in~\eqref{eq:Tg_SL},
this implies that the first order tangential homological equation~\eqref{eq:conj_tangent} is equivalent to the equation
\begin{equation} \label{eq:conj_tangent_sl_1}
\begin{aligned}
    \partial_{\omega} \ord{g}{1}(\phi) + \ord{f}{1}(\phi) = 
     -\frac{\delta}{\gamma} &\begin{pmatrix}
         \cos(\phi_2 - \phi_1 - \Omega \tau + \cc) - \cos(\cc)  \\
       \cos(\phi_1 - \phi_2 - \Omega \tau + \cc) - \cos(\cc)
    \end{pmatrix} \\
    +
    &\begin{pmatrix}
        \sin(\phi_2 - \phi_1 - \Omega \tau + \cc) - \sin(\cc) \\
        \sin(\phi_1 - \phi_2 - \Omega \tau + \cc) - \sin(\cc) 
    \end{pmatrix}.
\end{aligned}
\end{equation}
The right hand side of the last equation is in normal form, in the sense that its Fourier series consists solely of modes $e^{i\dotp{(k_1, k_2), (\phi_1, \phi_2)}}$ with $\dotp{(k_1, k_2), \omega} = \dotp{(k_1, k_2), (\Omega, \Omega)} = 0$. 
Hence we can solve the equation~\eqref{eq:conj_tangent_sl_1} in such a way that $\ord{f}{1}$ is in normal form (cf.~Step~(ii) in Section~\ref{subsec:algorithm}) by setting 
\begin{subequations}
    \begin{equation}
  \ord{g}{1}(\phi) \equiv 0 
    \end{equation}
and 
\begin{equation} \label{eq:sl_f1}
\begin{aligned}
    \ord{f}{1}(\phi) = -\frac{\delta}{\gamma} &\begin{pmatrix}
         \cos(\phi_2 - \phi_1 - \Omega \tau + \cc) - \cos(\cc)  \\
       \cos(\phi_1 - \phi_2 - \Omega \tau + \cc) - \cos(\cc)
    \end{pmatrix} \\
    +
    &\begin{pmatrix}
        \sin(\phi_2 - \phi_1 - \Omega \tau + \cc) - \sin(\cc) \\
        \sin(\phi_1 - \phi_2 - \Omega \tau + \cc) - \sin(\cc) 
    \end{pmatrix}. 
\end{aligned}
\end{equation}
\end{subequations}
The expression for the first order vector field $\ord{f}{1}$ together with the zeroth order vector field $\ord{f}{0}$ (cf.~\eqref{eq:e0_sl}) now yields the statement of Lemma~\ref{lem:phasereductiontwoSLoscillators}. 

As a preparation for the second order computations in the next subsection, we here also solve the first order normal homological equation~\eqref{eq:conj_normal}. To that end, we observe that 
\[
(I - \pi(0)) \begin{pmatrix}
    z_1 \\ z_2
\end{pmatrix}
= 
(1 + i \frac{\delta}{\gamma})
\begin{pmatrix}
\Re(z_1) \\
\Re(z_2)
\end{pmatrix}
\]
from which we find that the right hand side of the first order normal homological equation~\eqref{eq:conj_normal} is given by 
\begin{align*}
    [I-\pi(\phi)] \ord{\eta}{1}(\phi)  = R(1 + i \frac{\delta}{\gamma})
\begin{pmatrix}
    e^{i \phi_1} \left[\cos(\phi_2 - \phi_1 - \Omega \tau + \cc) - \cos(\cc) \right] \\
     e^{i \phi_2} \left[\cos(\phi_1 - \phi_2 - \Omega \tau + \cc) - \cos(\cc) \right]
\end{pmatrix}.
\end{align*}
Combining this with the expression for $N(\phi)$ in~\eqref{eq:N_sl},
we find that the first order normal homological equation~\eqref{eq:conj_normal} is equivalent to the equation
\[
\partial_{\omega} \ord{h}{1}(\phi) - L \ord{h}{1}(\phi) = \frac{R}{\gamma} \begin{pmatrix}
    \cos(\phi_2 - \phi_1 - \Omega \tau + \cc) - \cos(\cc) \\
   \cos(\phi_1 - \phi_2 - \Omega \tau + \cc) - \cos(\cc) 
\end{pmatrix},
\]
which in turn is solved by 
\begin{equation} \label{eq:sl_h1}
    \ord{h}{1}(\phi) = \frac{R}{2 \gamma \alpha} \begin{pmatrix}
    \cos(\phi_2 - \phi_1 - \Omega \tau + \cc) - \cos(\cc) \\
   \cos(\phi_1 - \phi_2 - \Omega \tau + \cc) - \cos(\cc) 
\end{pmatrix}. 
\end{equation}
From here, we can compute the expressions for $\ord{e}{1}$ and $\ord{E}{1}$ from the equations \eqref{eq:ej_decomp} and \eqref{eq:Ej_phi_s}; we state the result of this computation for the specific parameter choice $\delta = 0$ in the next section.

\subsection{Second-order phase reduction}

We next compute the second order contribution contribution to the phase dynamics, i.e., we set $\el = 2$ in the tangential homological equation~\eqref{eq:conj_tangent} and solve for $\ord{f}{2}$ and $\ord{g}{2}$. To keep computations and the lengths of formulas manageable, we set $\delta = 0$ (this corresponds to the case where the isochrons of the uncoupled oscillators are straight lines~\cite{Leon2019a}). 
In the case $\delta = 0$, the expression for $\ord{f}{1}, \ord{e}{1}$ and $\ord{E}{1}$ from the previous section reduce to
\begin{equation}
\label{eq:info1_delta0}
    \begin{aligned} 
    \ord{f}{1}(\phi) &= \begin{pmatrix}
        \sin(\phi_2 - \phi_1 - \Omega \tau + \cc) - \sin(\cc) \\
        \sin(\phi_1 - \phi_2 - \Omega \tau + \cc) - \sin(\cc) 
    \end{pmatrix} \\
        \ord{e}{1}(\phi) &= \frac{R}{2 \alpha} \begin{pmatrix}
    e^{i \phi_1} \left[\cos(\phi_2 - \phi_1 - \Omega \tau + \cc) - \cos(\cc) \right] \\
   e^{i \phi_2} \left[\cos(\phi_1 - \phi_2 - \Omega \tau + \cc) - \cos(\cc)  \right]
\end{pmatrix} \\ 
\ord{E}{1}(\phi, s) &= \frac{R}{2 \alpha}
\begin{pmatrix}
    e^{i(\phi_1 + \Omega s)} \left[\cos(\phi_2 - \phi_1 - \Omega \tau + \cc) - \cos(\cc) \right] \\
    e^{i(\phi_2 + \Omega s)} \left[\cos(\phi_1 - \phi_2 - \Omega \tau + \cc) - \cos(\cc) \right]
\end{pmatrix}
 \\
 &\qquad+ R i s \begin{pmatrix}
     e^{i(\phi_1 + \Omega s)} \left[\sin(\phi_2 - \phi_1 - \Omega \tau + \cc) - \sin(\cc) \right] \\
     e^{i(\phi_2 + \Omega s)} \left[\sin(\phi_1 - \phi_2 - \Omega \tau + \cc) - \sin(\cc) \right]
 \end{pmatrix}
    \end{aligned}
\end{equation}

We next derive an explicit expression for the quantity~$\pi(\phi) \ord{\eta}{2}(\phi)$, which corresponds to the right hand side of the second order tangential homological equation~\eqref{eq:conj_tangent}. Using the expression for $\ord{\eta}{2}$ from~\eqref{eq:G1G2_expl}, we first write 
\begin{equation}
    \label{eq:pi_eta2}
    \begin{aligned}
\pi(\phi) \ord{\eta}{2}(\phi) &= \frac{1}{2}\pi(\phi)F''(\ord{e}{0}(\phi)) \left[\ord{e}{1}(\phi), \ord{e}{1}(\phi)\right]  
\\&\qquad - \pi(\phi) \left[\partial_\phi \ord{e}{1}(\phi) \cdot \ord{f}{1}(\phi)\right] \\&\qquad + \pi(\phi) \left[G'(\ord{E}{0}(\phi, \, .)) \ord{E}{1}(\phi, \, .) \right], 
\end{aligned}
\end{equation}
and we proceed by computing each of the three terms individually. 

We start by showing that the first term vanishes, i.e.,
\begin{equation} \label{eq:sl_G2term_zero}
    \frac{1}{2}\pi(\phi)F''(\ord{e}{0}(\phi)) \left[\ord{e}{1}(\phi), \ord{e}{1}(\phi)\right] = 0. 
\end{equation}
Given a point $z = (z_1, z_2) \in \mathbb{C}^2$ and a vector $v = (v_1, v_2) \in \mathbb{C}^2$,
we compute the second derivative of the uncoupled vector field $F$ (cf.~\eqref{eq:F_sl}) as
\[
    F''\begin{pmatrix}
        \comp{z}{1} \\
        \comp{z}{2}
    \end{pmatrix}
    \left[
    \begin{pmatrix}
        \comp{v}{1} \\ \comp{v}{2}
    \end{pmatrix}, 
    \begin{pmatrix}
        \comp{v}{1} \\ \comp{v}{2}
    \end{pmatrix}
    \right]= 
    \gamma 
    \begin{pmatrix}
        4 \comp{z}{1} |\comp{v}{1}|^2 + 2 \comp{\bar{z}}{1} (\comp{v}{1})^2 \\
        4 \comp{z}{2} |\comp{v}{2}|^2 + 2 \comp{\bar{z}}{2} (\comp{v}{2})^2
    \end{pmatrix}
\]
(here, we have used that $\delta = 0$). This implies that 
\begin{align*}
\frac{1}{2}F''(\ord{e}{0}(\phi))\left[\ord{e}{1}(\phi), \ord{e}{1}(\phi)\right] &= 
2 \gamma
\begin{pmatrix}
    e_1^{(0)} \left| e_1^{(1)}\right|^2 \\
   e_2^{(0)} \left| e_2^{(1)}\right|^2
\end{pmatrix}
+ \gamma 
\begin{pmatrix}
    \bar{e}_1^{(0)} \left( e_1^{(1)}\right)^2 \\
   \bar{e}_2^{(0)} \left( e_2^{(1)}\right)^2
\end{pmatrix} 
\end{align*}
We rewrite the right hand side as 
\[
2 \gamma R
\begin{pmatrix}
    e^{i \phi_1} \left| e_1^{(1)}\right|^2 \\
    e^{i \phi_2} \left| e_2^{(1)}\right|^2
\end{pmatrix}
+ \gamma \frac{R^3}{4 \alpha^2} \begin{pmatrix}
    e^{i \phi_1} \left[\cos(\phi_2- \phi_1 - \Omega \tau + \cc) - \cos(\cc)\right]^2 \\
    e^{i\phi_2}  \left[\cos(\phi_1- \phi_2 - \Omega \tau + \cc) - \cos(\cc)\right]^2
\end{pmatrix}
\]
and see that this expression lies in the kernel of the projection operator~$\pi(\phi)$
(cf.~\eqref{eq:sl_pi} with $\delta = 0$), from which the identity
\eqref{eq:sl_G2term_zero} follows.  

We next compute the second term of $\pi(\phi) \ord{\eta}{2}(\phi)$ in~\eqref{eq:pi_eta2}, i.e., we compute $\pi(\phi) \left[\partial_\phi \ord{e}{1}(\phi) \cdot \ord{f}{1}(\phi)\right]$. 
Write
$\Delta_1 = \phi_2 - \phi_1 - \Omega \tau + \cc$,
$\Delta_2 = \phi_1 - \phi_2 - \Omega \tau + \cc$
for the shifted phase differences.
Using the expressions for the first order contributions~$\ord{e}{1}$ and~$\ord{f}{1}$ in~\eqref{eq:info1_delta0}, we see that 
$\partial_\phi \ord{e}{1}(\phi) \cdot \ord{f}{1}(\phi)$ is of the form 
\[
\partial_\phi \ord{e}{1}(\phi) \cdot \ord{f}{1}(\phi)
=
\begin{pmatrix}
    e^{i \phi_1} z_1 \\
    e^{i \phi_2} z_2
\end{pmatrix}
\]
where $z_1, z_2$ are two complex numbers whose imaginary parts are given by
\begin{align*}
    \Im(z_1) &= \frac{R}{2 \alpha}\left( \cos(\Delta_1) - \cos(\cc) \right)\left( \sin(\Delta_1) - \sin(\cc) \right) \\
    \Im(z_2) &= \frac{R}{2 \alpha}\left( \cos(\Delta_2) - \cos(\cc) \right)\left( \sin(\Delta_2) - \sin(\cc) \right)
\end{align*}
Using again the projection operator $\pi(\phi)$ in~\eqref{eq:sl_pi} (with $\delta = 0$), we 
derive that 
\begin{equation} \label{eq:sl_pi_e1_f1}
\begin{aligned}
    \pi(\phi)\partial_\phi \ord{e}{1}(\phi) \cdot \ord{f}{1}(\phi)
&= \frac{iR}{4 \alpha}\sin(2\cc)  \begin{pmatrix}
        e^{i \phi_1}\\
        e^{i \phi_2}
    \end{pmatrix}  
    \\&\qquad 
    - \frac{iR}{2 \alpha} \cos(\cc)
    \begin{pmatrix}
        e^{i \phi_1} \sin(\Delta_1) \\
        e^{i \phi_1} \sin(\Delta_2)
    \end{pmatrix}
    \\&\qquad 
    - \frac{iR}{2 \alpha} \sin(\cc)
    \begin{pmatrix}
        e^{i \phi_1} \cos(\Delta_1) \\
        e^{i \phi_2} \cos(\Delta_2)
    \end{pmatrix}  
    \\&\qquad 
    + \frac{i R}{4 \alpha}
    \begin{pmatrix}
        e^{i\phi_1}\sin(2\Delta_1) \\
        e^{i\phi_2}\sin(2\Delta_2)
    \end{pmatrix} .
\end{aligned}
\end{equation}

We finally compute the third term in the expression for $\pi(\phi) \ord{\eta}{2}(\phi)$ in~\eqref{eq:pi_eta2}, i.e., we compute $\pi(\phi) \left[G'(\ord{E}{0}(\phi, \, .)) \ord{E}{1}(\phi, \, .) \right]$. The functional~$G$ defined in~\eqref{eq:G_sl} is linear and hence
\begin{align*}
G'(\ord{E}{0}(\phi, \, .)) \ord{E}{1}(\phi, \, .) &= G(\ord{E}{1}(\phi, \, .)) = \begin{pmatrix}
    e^{i \cc} \left[\ord{E}{1}_2(\phi, -\tau) - \ord{e}{1}_1(\phi) \right] \\
    e^{i \cc} \left[\ord{E}{1}_1(\phi, -\tau) - \ord{e}{1}_2(\phi) \right]
\end{pmatrix}  \\
&= 
\frac{R}{2 \alpha} \begin{pmatrix}
    e^{i (\phi_2 - \Omega \tau +\cc)} \left[\cos(\Delta_2) - \cos(\cc) \right] \\
    e^{i (\phi_1 - \Omega \tau+\cc)} \left[\cos(\Delta_1) - \cos(\cc) \right]
\end{pmatrix} 
\\&\qquad
- R i \tau 
\begin{pmatrix}
    e^{i (\phi_2 - \Omega \tau+\cc)} \left[\sin(\Delta_2)- \sin(\cc) \right] \\
    e^{i (\phi_1 - \Omega \tau+\cc)} \left[\sin(\Delta_1)- \sin(\cc) \right]
\end{pmatrix} 
\\&\qquad 
-\frac{R}{2 \alpha} \begin{pmatrix}
    e^{i (\phi_1+\cc)} \left[\cos(\Delta_1) - \cos(\cc) \right] \\
   e^{i (\phi_2+\cc)} \left[\cos(\Delta_2) - \cos(\cc)  \right]
   \end{pmatrix}.
\end{align*}
Together with the expression for~$\pi(\phi)$ (cf.~\eqref{eq:sl_pi} with $\delta = 0$) this implies that
\begin{align*}
    \pi(\phi) G'(\ord{E}{0}(\phi, \, .)) \ord{E}{1}(\phi, \, .) &= \frac{i R}{2 \alpha} \begin{pmatrix}
        e^{i \phi_1} \sin(\Delta_1) \left[\cos(\Delta_2) - \cos(\cc) \right] \\
        e^{i \phi_2} \sin(\Delta_2) \left[\cos(\Delta_1) - \cos(\cc) \right]
    \end{pmatrix} 
    \\&\qquad
    - i R \tau \begin{pmatrix}
     e^{i  \phi_1} \cos(\Delta_1) \left[\sin(\Delta_2) - \sin(\cc) \right] \\
     e^{i  \phi_2} \cos(\Delta_2) \left[\sin(\Delta_1) - \sin(\cc) \right] 
 \end{pmatrix}
 \\&\qquad
 - \frac{i R}{2 \alpha}
     \begin{pmatrix}
         e^{i \phi_1} \sin(\cc) \left[\cos(\Delta_1)- \cos(\cc)\right] \\
         e^{i \phi_2} \sin(\cc) \left[\cos(\Delta_2)-\cos(\cc) \right]
     \end{pmatrix}
\end{align*}
To simplify the above expression, we let $x \in \mathbb{R}$ and use the double angle formula's to derive that 
\begin{align*}
    \sin(x) \cos(-x+2\cc - 2 \Omega \tau) &= \sin(x) \left(\cos(x) \cos(2\cc - 2 \Omega \tau) + \sin(x) \sin(2\cc - 2 \Omega \tau)\right) \\
    & = \frac{1}{2}\cos(2\cc - 2 \Omega \tau) \sin(2x) + \frac{1}{2} \sin(2\cc - 2 \Omega \tau) \left(1 - \cos(2x) \right) \\
    \cos(x) \sin(-x+2\cc - 2 \Omega \tau) &= \cos(x) \left( -\sin(x) \cos(2\cc - 2 \Omega \tau) + \cos(x) \sin(2\cc - 2 \Omega \tau) \right) \\
    & = - \frac{1}{2} \cos(2\cc - 2 \Omega \tau)\sin(2x)  + \frac{1}{2} \sin(2\cc - 2 \Omega \tau) (1 + \cos(2x)).
\end{align*}
We apply these identities with $x = \Delta_1$ to the first component of the vector $\pi(\phi) G'(\ord{E}{0}(\phi, \, .)) \ord{E}{1}(\phi, \, .)$; and apply the identities with $x = \Delta_2$ to the second component of the vector $\pi(\phi) G'(\ord{E}{0}(\phi, \, .)) \ord{E}{1}(\phi, \, .)$; and 
we find that
\begin{align*}
    \pi(\phi) G'(\ord{E}{0}(\phi, \, .)) \ord{E}{1}(\phi, \, .) &= \left[\frac{iR}{4 \alpha} \sin(2\cc) +\left(\frac{iR}{4 \alpha} - \frac{iR \tau}{2}\right) \sin(2\cc - 2 \Omega \tau) \right]
    \begin{pmatrix}
        e^{i \phi_1}  \\
        e^{i \phi_2}
    \end{pmatrix}
    \\&\qquad
    - \frac{iR}{2 \alpha} \cos(\cc)
    \begin{pmatrix}
    e^{i \phi_1} \sin(\Delta_1) \\  
    e^{i \phi_2} \sin(\Delta_2)
    \end{pmatrix} 
\\&\qquad
    + \left[ - \frac{iR}{2 \alpha} + i R \tau \right] \sin(\cc)
    \begin{pmatrix}
        e^{i \phi_1}  \cos(\Delta_1) \\
        e^{i \phi_2}  \cos(\Delta_2)
    \end{pmatrix} 
    \\&\qquad
    +\left[ \frac{iR}{4 \alpha} + \frac{iR \tau}{2} \right] \cos(2\cc -2 \Omega \tau) 
    \begin{pmatrix}
        e^{i \phi_1} \sin(2\Delta_1) \\
        e^{i \phi_2} \sin(2\Delta_2)
    \end{pmatrix} 
    \\&\qquad
    + \left[ - \frac{iR}{4 \alpha} - \frac{iR\tau}{2} \right] \sin(2\cc - 2 \Omega \tau) 
    \begin{pmatrix}
        e^{i \phi_1} \cos(2\Delta_1) \\
        e^{i \phi_2} \cos(2\Delta_2)
    \end{pmatrix}. 
\end{align*}

Together with the expressions for the first two components of $\pi(\phi) \ord{\eta}{2}$ computed in~\eqref{eq:sl_pi_e1_f1} and~\eqref{eq:sl_G2term_zero}, we can now compute $\pi(\phi) \ord{\eta}{2}(\phi)$ from~\eqref{eq:pi_eta2}, i.e., we can compute the right hand side of the second order tangential homological equation~\eqref{eq:conj_tangent}. Next using that $\Tg(\phi) = \mbox{diag}(iR e^{i \phi_1}, iR e^{i \phi_2})$ (cf.~\eqref{eq:Tg_SL}), we derive that the second tangential homological equation~\eqref{eq:conj_tangent} is equivalent to the equation 
\begin{align*}
    \partial_{\omega} \ord{g}{2}(\phi) + \ord{f}{2}(\phi) &= \left[\frac{1}{4 \alpha} - \frac{\tau}{2}\right] \sin(2\cc - 2 \Omega \tau) 
    \begin{pmatrix}
        1  \\
        1
    \end{pmatrix}
    \\&\qquad
    + \  \tau  \sin(\cc)
    \begin{pmatrix}
         \cos(\Delta_1) \\
          \cos(\Delta_2)
    \end{pmatrix}
    \\&\qquad
    + 
    \left[ \left[\frac{1}{4 \alpha} + \frac{\tau}{2} \right] \cos(2 \cc - 2 \Omega \tau) - \frac{1}{4 \alpha} \right]
    \begin{pmatrix}
        \sin(2\Delta_1) \\
        \sin(2\Delta_2)
    \end{pmatrix}
    \\&\qquad
    - \left[ \frac{1}{4 \alpha} +\frac{\tau}{2} \right] \sin(2\cc - 2 \Omega \tau) 
    \begin{pmatrix}
         \cos(2\Delta_1) \\
         \cos(2\Delta_2)
    \end{pmatrix}. 
\end{align*}
The right hand side of the last equation is in normal form, i.e., the Fourier series of the right hand side consists solely of Fourier modes $e^{i\dotp{(k_1, k_2), (\phi_1, \phi_2)}}$ with $\dotp{(k_1, k_2), \omega} = \dotp{(k_1, k_2), (\Omega, \Omega)} = 0$. 
We can therefore solve the equation in such a way that $\ord{f}{2}$ is in normal form (cf.~Section~\ref{subsec:algorithm}) by setting 
\[
\ord{g}{2} = 0
\]
and 
\begin{align*}
    \ord{f}{2}(\phi) &= \left[\frac{1}{4 \alpha} - \frac{\tau
    }{2}\right] \sin(2\cc - 2 \Omega \tau) 
    \begin{pmatrix}
        1  \\
        1
    \end{pmatrix}\\
    &\qquad+ \  \tau  \sin(\cc)
    \begin{pmatrix}
        \cos(\phi_2 - \phi_1 - \Omega \tau + \cc) \\
        \cos(\phi_1 - \phi_2 - \Omega \tau + \cc)
    \end{pmatrix} \\
    &\qquad+ 
    \left[ \left[\frac{1}{4 \alpha} + \frac{\tau}{2} \right] \cos(2 \cc - 2 \Omega \tau) - \frac{1}{4 \alpha} \right]
    \begin{pmatrix}
        \sin(2(\phi_2 - \phi_1 - \Omega \tau + \cc)) \\
        \sin(2(\phi_1 - \phi_2 - \Omega \tau + \cc))
    \end{pmatrix} \\
    &\qquad- \left[ \frac{1}{4 \alpha} +\frac{\tau}{2} \right] \sin(2\cc - 2 \Omega \tau) 
    \begin{pmatrix}
        \cos(2(\phi_2 - \phi_1 - \Omega \tau + \cc)) \\
        \cos(2(\phi_1 - \phi_2 - \Omega \tau + \cc))
    \end{pmatrix} . 
\end{align*}
Together with the expressions for the zeroth order vector field $\ord{f}{0}$ and the first order vector field $\ord{f}{1}$ computed in the previous subsection, this now yields the statement of Lemma~\ref{lem:SLso}. 

\section{Synchronization and bistability of delay-coupled Stuart--Landau oscillators}
\label{sec:bistab}

In this section, we use the phase reduced system to deduce information about the delay-coupled Stuart--Landau oscillators~\eqref{eq:sl_coupled}. 
Recall from Lemma~\ref{lem:SLso} that for this system with parameter $\delta = 0$, the second order phase reduction in terms of the phase difference $\psi := \phi_1 - \phi_2$ is given by
\begin{equation} \label{eq:psi_ord2}
\begin{aligned}
     \dot{\psi} &=  -2 \epsilon \cos(\cc- \omega \tau) \sin(\psi)  \\
    &\qquad + \epsilon^2 \left( 2 \tau \sin(\cc) \sin(\cc - \omega \tau) \sin(\psi) - \left[\tau + \frac{1}{ \alpha} \sin^2(\cc -  \omega  \tau) \right]\sin(2\psi) \right)\\
    &\qquad + \mathcal{O}(\varepsilon^3).
\end{aligned}
\end{equation}
First, we prove that for small coupling strengths and small delays, there exists an open region in parameter space for which in-phase synchrony (the equilibrium $\psi=0$) and anti-phase synchrony (the equilibrium $\psi=\pi$) are simultaneously dynamically stable.
Second, we then compare the stability boundaries obtained from the one-dimensional phase reduction to numerical solutions of the infinite-dimensional delay-coupled Stuart-Landau system~\eqref{eq:sl_coupled}.
Even for larger delays, the phase reduction predicts these stability boundaries (and hence also the regions of bistability) remarkably well.

Linear stability of in-phase and anti-phase synchrony for the second-order phase dynamics~\eqref{eq:psi_ord2} are determined by the eigenvalues. 
\begin{subequations}
\begin{align}\label{eq:lambdasync}
\lambda_\mathrm{sync} &= -2\eps \cos(\cc-\omega\tau)-2\eps^2\left(\!\tau+\frac{1}{\alpha}\sin^2(\cc-\omega\tau)\!-\!\tau\sin(\cc)\sin(\cc-\omega\tau)\!\right)+ \mathcal{O}(\varepsilon^3) \\
\lambda_\mathrm{splay} \! &= \!  2 \eps \cos(\cc-\omega\tau) \! -\! 2 \eps^2\left(\!\tau \! +\! \frac{1}{\alpha}\sin^2(\cc-\omega\tau)\!+\!\tau\sin(\cc)\sin(\cc-\omega\tau)\!\right) + \mathcal{O}(\varepsilon^3). \label{eq:lambdasplay}
\end{align}
\end{subequations}
The eigenvalues vanish near the points $(\cc, \eps, \tau) = (\frac{\pi}{2}, 0,0)$ and $(\cc, \eps, \tau) = (\frac{3\pi}{2}, 0,0)$ in parameter-space.
Considering only first-order terms yields a degenerate bifurcation scenario that is completely determined by the effective phase shift parameter $\rho-\omega\tau$: For $-\frac{\pi}{2} < \rho - \omega\tau < \frac{\pi}{2}$ in-phase synchrony is linearly stable and anti-phase is unstable, for $\frac{\pi}{2} < \rho - \omega\tau < \frac{3\pi}{2}$ stability is inverted, and exchange of stability happens at $\rho-\omega\tau \in \{\frac{\pi}{2},\frac{3\pi}{2}\}$.
The following proposition provides an  approximation of where these eigenvalues vanish near the points $(\cc, \eps, \tau) = (\frac{\pi}{2}, 0,0)$ and $(\cc, \eps, \tau) = (\frac{3\pi}{2}, 0,0)$ in parameter-space based on the second order phase approximation~\eqref{eq:psi_ord2}:

\begin{prop}\label{prop:expansions}
Fix the parameters $\alpha, \beta, \gamma \neq 0$ and $\delta=0$ of the uncoupled Stuart-Landau oscillators~\eqref{eq:sl_single}. There exist a relatively open neighborhood $U\subset \mathbb{R}/2\pi\mathbb{Z} \times [0,\infty) \times [0,\infty)$ of  $(\frac{\pi}{2},0,0)$, a relatively open neighborhood $V \subset [0,\infty) \times [0,\infty)$ of $(0, 0)$, and a smooth function $\cc_{\rm sync}^{\pi/2}: V\to \mathbb{R}/2\pi\mathbb{Z}$ such that
$$\{\ (\rho, \eps, \tau)\in U \ | \ \lambda_{\rm sync} = 0\ \} = \{\ (\cc_{\rm sync}^{\rm \pi/2} (\eps, \tau), \eps, \tau)\in \mathbb{R}^3 \ | \ (\eps, \tau)\in V \ \} \, . $$
 The Taylor expansion of this function near $(\eps, \tau)=(0,0)$ is given by
    \begin{align}\label{eq:rhosyncpiovertwoexpansion}
    \cc_\mathrm{sync}^{\pi/2}(\eps, \tau)  & = \frac{\pi}{2} + \omega \tau + \varepsilon\alpha^{-1} + \frac{1}{2}\varepsilon \omega^2 \tau^3  + \mathcal{O}(\eps^2 + \varepsilon \tau^5)\, .
 \end{align}
 The zero set of $\lambda_{\rm splay}$ near $(\rho, \eps, \tau)=(\frac{\pi}{2}, 0, 0)$ and the zero sets of $\lambda_{\rm sync}$ and $\lambda_{\rm splay}$  near $(\rho, \eps, \tau)=(\frac{3\pi}{2}, 0, 0)$ are given by similar smooth functions, with Taylor expansions
    \begin{align}
     \cc_\mathrm{splay}^{\pi/2}(\eps, \tau)  & = \frac{\pi}{2}  + \omega\tau - \varepsilon \alpha^{-1} - 2 \eps\tau  + \frac{1}{2}\eps \omega^2 \tau^3 +  \mathcal{O}(\eps^2 + \varepsilon \tau^5)\, . \\ 
     \cc_\mathrm{sync}^{3\pi/2}(\eps, \tau)  & = \frac{3\pi}{2} + \omega \tau  - \varepsilon \alpha^{-1} - \frac{1}{2}\eps \omega^2 \tau^3  + \mathcal{O}(\eps^2 + \varepsilon \tau^5)\, .\\
     \cc_\mathrm{splay}^{3\pi/2}(\eps, \tau)  & = \frac{3\pi}{2} + \omega\tau  + \varepsilon \alpha^{-1} + 2\eps \tau - \frac{1}{2} \eps\omega^2 \tau^3+ \mathcal{O}(\eps^2 + \varepsilon \tau^5)\, . 
    \end{align}
\end{prop}  
\begin{proof}
    Fixing $\alpha, \beta, \gamma\neq 0$ and $\delta=0$, we denote by  $\lambda_\mathrm{sync}(\rho, \eps, \tau)$ the eigenvalue given in~\eqref{eq:lambdasync} as a function of the remaining parameters $\rho, \eps, \tau$ in the  equations of motion for the delay-coupled Stuart-Landau operators~\eqref{eq:sl_coupled}. Now we define $$F(r, \eps, \tau) := \eps^{-2} \lambda_\mathrm{sync}\left(\frac{\pi}{2} + \omega \tau + \eps r, \eps, \tau\right)\ .$$ 
     Then $F$ can be extended to a  smooth  function in an open neighborhood of the point $(r,\eps,\tau)=(0,0,0)$. We have $F(r, 0,0) =  2r$ and therefore that $F(0,0,0)=0$ and $\frac{\partial F}{\partial r} (r,0, 0)=  2 \neq 0$. It thus follows from the implicit function theorem that there exists a function  $r_{\rm sync}^{\rm \pi/2} = r_{\rm sync}^{\rm \pi/2}(\eps, \tau)$ defined in an open neighborhood of $(\eps,\tau)=(0,0)$,  satisfying $r_{\rm sync}^{\rm \pi/2}(0,0)=0$ and  
   $F(r_{\rm sync}^{\rm \pi/2}(\eps, \tau), \eps, \tau)=0$.
   
  The Taylor expansion of this function can determined to any order by implicit differentiation. For example, we have
  $ F(r, \eps, \tau) = 2r  - \frac{2}{\alpha} - \omega^2 \tau^3 +\mathcal{O}(\tau^5)+\mathcal{O}(\eps)$ and therefore $r_{\rm sync}^{\pi/2}(\eps, \tau) = \frac{1}{\alpha} + \frac{1}{2}\omega^2\tau^3 + \mathcal{O}(\tau^5)+\mathcal{O}(\eps)$. This proves formula~\eqref{eq:rhosyncpiovertwoexpansion} for the Taylor expansion of $\rho_{\rm sync}^{\pi/2} = \frac{\pi}{2}  + \omega\tau + \eps r_{\rm sync}^{\pi/2}$. 
   The claims about the existence and Taylor expansions of $\rho_{\rm splay}^{\pi/2}, \rho_{\rm sync}^{3\pi/2}$ and $\rho_{\rm splay}^{3\pi/2}$   are proved in a similar way.
\end{proof}

\begin{remark}
    Proposition~\ref{prop:expansions} describes for which parameters the eigenvalues $\lambda_{\rm sync}$ and $\lambda_{\rm splay}$ vanish, but it does not state explicitly when these eigenvalues are positive or negative. The signs of the eigenvalues are not hard to determine though. Consider for example the eigenvalue $\lambda_{\rm sync}$ of the synchronous state near $(\rho, \eps, \tau)=(\frac{\pi}{2},0,0)$, which we studied in detail in the proof of the proposition. Because $\frac{\partial F}{\partial r}(r,0,0)=2>0$, it is clear that $F(r,\eps, \tau)>0$ for $r>r_{\rm sync}^{\pi/2}(\eps, \tau)$ and $F(r,\eps, \tau)<0$ for $r<r_{\rm sync}^{\pi/2}(\eps, \tau)$. This shows that $\lambda_{\rm sync}>0$ when $\rho > \rho_{\rm sync}^{\pi/2}(\eps, \tau)$ and $\lambda_{\rm sync}<0$ when $\rho < \rho_{\rm sync}^{\pi/2}(\eps, \tau)$. Similarly, $\lambda^{\pi/2}_{\rm splay} >0$ precisely when $\rho < \rho^{\pi/2}_{\rm splay}(\eps, \tau)$  and $\lambda^{\pi/2}_{\rm splay} < 0$ precisely when $\rho > \rho^{\pi/2}_{\rm splay}(\eps, \tau)$. Analogous statements hold for the parameters near $(\frac{3\pi}{2}, 0, 0)$. 
\end{remark} 
As a corollary of Proposition~\ref{prop:expansions} we find that there are regions of bistability in parameter space.

\begin{cor}\label{cor:Bistab}
    There exists a relatively open set of parameters $\tilde U$ in $(0, 2\pi)\times[0,\infty)\times[0,\infty)$, containing $(\frac{\pi}{2},0,0)$ in its closure, such that for $(\rho, \eps, \tau)\in 
    \tilde U$,  the synchronous state $\psi = 0$ 
    and the splay state $\psi=\pi$ 
     are both locally exponentially stable for the phase reduced equation~\eqref{eq:psi_ord2}. 
\end{cor}
\begin{proof}
Let $U\subset \mathbb{R}/2\pi\mathbb{Z} \times [0,\infty) \times [0,\infty)$ be the relatively open neighborhood of $(\frac{\pi}{2},0,0)$ in the conclusion of 
 Proposition~\ref{prop:expansions}, and let $(\rho, \varepsilon, \tau)\in U$.  Then the synchronous state $\psi=0$ is exponentially stable under~\eqref{eq:psi_ord2}  if and only if $\lambda_{\rm sync} <0$, whereas the splay state $\psi=\pi$ is exponentially stable if and only if $\lambda_{\rm splay} <0$. The former happens precisely when $\rho < \rho_{\rm sync}^{\pi/2}(\eps, \tau)$ and the latter when  $\rho > \rho_{\rm splay}^{\pi/2}(\rho, \varepsilon, \tau)$, so the
  states are both exponentially stable when
  $$\rho_{\rm splay}^{\pi/2}(\eps, \tau) <\rho < \rho_{\rm sync}^{\pi/2}(\eps, \tau)\, .$$
 It follows from Proposition~\ref{prop:expansions} that there exists a constant $C>0$ such that 
 $$\rho_{\rm sync}^{\pi/2}(\eps, \tau) > \frac{\pi}{2}+\omega\tau + \varepsilon \alpha^{-1} - C(\eps^2+\eps\tau)\ \mbox{and}\ \rho_{\rm splay}^{\pi/2}(\eps, \tau) < \frac{\pi}{2}+\omega\tau - \eps \alpha^{-1} +  C (\eps^2+\eps\tau) \, .$$ 
By shrinking the sets $U, V$ in the conclusion of Proposition~\ref{prop:expansions} if necessary, we can arrange that this inequality holds for all $(\eps, \tau)\in V$. 
Now we define the  subset $$\tilde U : = \left\{ (\rho, \eps, \tau) \in U \ |\ 0< \varepsilon + \tau < \frac{1}{2C\alpha}  \, , \, \frac{\pi}{2} + \omega \tau - \frac{\varepsilon}{2\alpha} <\rho <  \frac{\pi}{2} + \omega \tau  + \frac{\varepsilon}{2\alpha} \right\} \subset U\, . $$ 
The set $\tilde U$ is nonempty and open and contains $(\frac{\pi}{2}, 0,0)$  in its closure. For $(\rho, \eps, \tau)\in \tilde U$ we have $C(\eps^2+\varepsilon \tau) < \frac{\varepsilon}{2\alpha}$, and therefore that
\begin{align}
\rho_{\rm splay}^{\pi/2}(\eps, \tau) & < \frac{\pi}{2}+ \omega\tau - \frac{\eps}{ \alpha}  +  C (\eps^2+\eps\tau)    \nonumber \\ \nonumber 
& < \frac{\pi}{2} + \omega \tau -\frac{\varepsilon}{2\alpha}   < \rho 
 < \frac{\pi}{2} + \omega \tau +\frac{\varepsilon}{2\alpha}     \\ \nonumber &  <  \rho_{\rm sync}^{\pi/2}(\eps, \tau) +C(\eps^2+\eps\tau) - \frac{\varepsilon}{2\alpha}  <  \rho_{\rm sync}^{\pi/2}(\eps, \tau)\, . 
\end{align}
This proves that for $(\rho, \eps, \tau)\in\tilde U$, the synchronous state and the splay state are both exponentially stable. 
  The analysis for $(\rho, \eps, \tau)$ near $(\frac{3\pi}{2},0,0)$ is similar.
\end{proof}

The one-dimensional phase-reduced dynamics~\eqref{eq:psi_ord2} predict the dynamics of the infinite-dimensional unreduced dynamics.
To compare the systems, we computed numerical solutions of the delay-coupled system~\eqref{eq:sl_coupled} for varying parameters using \verb|dde23| implemented in \textsc{Matlab} with random initial condition such that the phase difference was sampled uniformly on~$[0, 2\pi)$.
The coloring in Figure~\ref{fig:CompDDEPhaseRed} shows regions in the $(\rho, \tau)$-parameter plane where only in-phase synchrony is stable (blue) and regions where only anti-phase synchrony is stable (red). 
Region speckled in red and blue correspond to bistability of both synchronized solutions.

\begin{figure}
    \centering
    \includegraphics[width=1.0\linewidth]{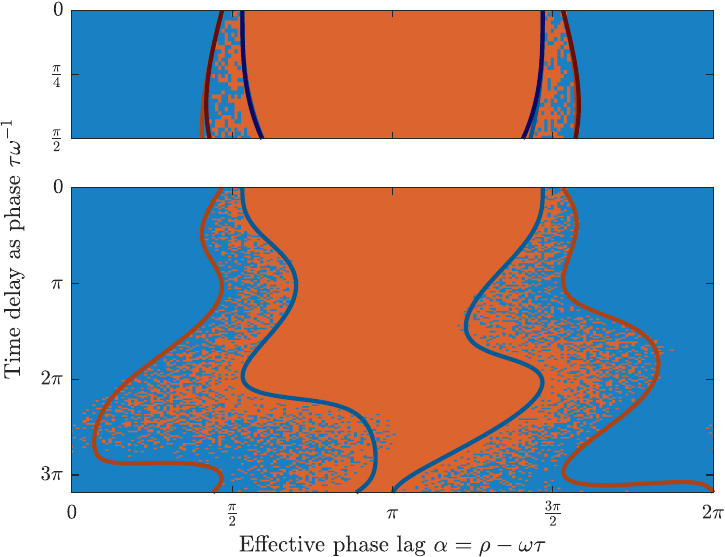}
    \caption{The bifurcation curves of the one-dimensional phase reduced equations~\eqref{eq:psi_ord2}---computed both analytically (light colored) and numerically (dark colored)---reflect the infinite-dimensional dynamics of the delay-coupled Stuart--Landau oscillators~\eqref{eq:sl_coupled} for $\eps=0.1$.
    The background color corresponds to $\Psi = \arg(z_1(T))-\arg(z_2(T))$ for numerical solutions~$z(t)$ of the DDE after $T=1000$ time units, for initial conditions chosen with a  uniformly random phase difference $\Psi(0)$; blue corresponds to in-phase synchrony ($\Psi=0$), red corresponds to anti-phase synchrony ($\Psi=\pi$). 
    The top panel shows the analytic approximation of the bifurcation lines of in-phase synchrony $\psi=0$ (dark blue line) and anti-phase synchrony $\psi=\pi$ (dark red line) predicted from the phase-reduced dynamics~\eqref{eq:psi_ord2} in Proposition~\ref{prop:expansions}.
    For comparison, the bifurcation lines obtained by numerically solving the bifurcation equations~\eqref{eq:lambdasync}--\eqref{eq:lambdasplay}
    for in-phase synchrony (blue line) and anti-phase synchrony (red line) only deviate slightly from the analytical approximations.
    The bottom panel shows the numerically computed bifurcation lines for a larger range of time delays.
    These indicate that the second-order phase reduction gives a reasonable approximation of the infinite-dimensional delay equation even for larger delays.
    }
    \label{fig:CompDDEPhaseRed}
\end{figure}

The approximate bifurcation boundaries computed in Proposition~\ref{prop:expansions} give a good approximation of the bifurcations of the DDE for small delays; see the top panel in Figure~\ref{fig:CompDDEPhaseRed}.
Around $\alpha = \rho-\Omega \tau$, the Taylor  approximation~$\cc_\mathrm{sync}^{\pi/2}$ as specified in~\eqref{eq:rhosyncpiovertwoexpansion} gives the right boundary of the bistability region (where in-phase synchrony loses stability), while~$\cc_\mathrm{splay}^{\pi/2}$ approximates the left boundary; a similar observation can be made near $\alpha=\frac{3\pi}{2}$.
The parameter regions, where Corollary~\ref{cor:Bistab} guarantees bistability of in-phase and anti-phase synchrony for the finite-dimensional phase equations, lie between the red and blue curves.
In this parameter range, the unreduced infinite-dimensional system also shows bistability between in-phase and anti-phase synchrony.

But even for larger delays and reasonable strong coupling the phase reduction yields a good approximation for the infinite-dimensional dynamics; cf.~bottom panel in Figure~\ref{fig:CompDDEPhaseRed}.
The blue and red lines correspond to the bifurcation lines of in-phase and anti-phase synchrony computed numerically from~\eqref{eq:lambdasync}.
They capture the dynamics of the unreduced system qualitatively for time delays exceeding one period even at a moderate coupling strength of $\eps=0.1$.

\section{Discussion and outlook} \label{sec:discussion}

Our approach provides a systematic method for phase reduction of delay-coupled oscillator networks.
The key idea is to simultaneously compute an approximation of the invariant torus in the infinite-dimensional phase space (through~$\ord{e}{\epsilon},\ord{E}{\epsilon}$) and the phase reduced dynamics~$\ord{f}{\epsilon}$.
The result is a system that approximates the dynamics on the entire invariant torus rather than in a neighborhood of a specific solution (such as the in-phase synchronized solution).
Computing the embeddings~$\ord{e}{\epsilon},\ord{E}{\epsilon}$ also has the advantage that we obtain phase/amplitude relationships. 
Specifically, for a given phase, the point~$\ord{e}{\epsilon}(\phi)\in \R^n$ gives the physical state of the system. Our method relies on the normal hyperbolicity of the invariant torus; if this assumption is not satisfied as the coupling strength~$\eps$ is increased, then incorporating additional degrees of freedom~\cite{Kotani2020} is required to capture the synchronization dynamics of the full system.

It would be desirable to extend our analysis beyond finite-dimensional oscillators with delay-coupling to oscillators with intrinsic delay.
In this case, the uncoupled dynamics are already infinite-dimensional.
The main technical challenge to be overcome is that the Floquet spectrum of the uncoupled dynamics is then infinite.

But even without delay in the nodes, the computations of the phase reduction to higher order becomes increasingly complex for larger networks of more general oscillators.
A computer implementation of the algorithm presented here would allow us to compute the phase reduction for general models in a systematic way---this is work in progress~\cite{Wedgwood2x}.
For scaling to large network sizes that are relevant in many real-world networks, a computer implementation would also benefit from decomposing the network into motifs.
For example, one may be able to compute phase reductions independently for specific motifs and {then combine the results.}
Such phase reduction techniques will allow to analyze the synchronization dynamics of coupled oscillators, for example, by elucidating how time delays enter nonpairwise phase interactions~\cite{Bick2021,Battiston2020,Battiston2025} as done for coupling without delay~\cite{Bick2023}.

To conclude, we comment on a functional analytic question that arises in our approach, namely the choice of underlying state space in Step~1 of Section~\ref{sec:method}.
There, we argued that every solution of the DDE~\eqref{eq:DDE} is also a solution of the coupled ODE-transport system~\eqref{eq:coupled_DDE}--\eqref{eq:coupled_DDE_bc}, and have subsequently \emph{chosen} the coupled system~\eqref{eq:coupled_DDE} as the starting point of our further analysis. However, we have not chosen a state space on which we consider the coupled system~\eqref{eq:coupled_DDE}. This is because there is no single state space for~\eqref{eq:coupled_DDE} that satisfies all our requirements, i.e. there is no single space on which i) every solution of the ODE-transport system~\eqref{eq:coupled_DDE}--\eqref{eq:coupled_DDE_bc} is also a solution of the DDE~\eqref{eq:DDE} and ii) conjugating the flow of the reduced dynamics~\eqref{eq:phase_intro} to the flow of the ODE-transport system~\eqref{eq:coupled_DDE}--\eqref{eq:coupled_DDE_bc} leads to equations that are practically tractable. 

If we require that every solution of~\eqref{eq:coupled_DDE}--\eqref{eq:coupled_DDE_bc} is also a solution of~\eqref{eq:DDE}, we need to choose a state space where each element is of the form $(c, \psi)$, with $c \in \mathbb{R}^n$ and $\psi: [-\tau, 0) \to \mathbb{R}^n$ a function that satisfies $c = \lim_{s \uparrow 0} \psi(s)$ (a concrete example is the space $C([-\tau, 0], \mathbb{R}^n) \simeq \{(c, \psi) \mid \lim_{s \uparrow 0} \psi(s) \}$). On such a space, one can show~\cite{diekmann2012} that the generator of the dynamics of~\eqref{eq:coupled_DDE} will have a domain that is dependent on $\epsilon$.  
Hence on such a state space, writing down a semi-conjugacy between the flow of the reduced problem~\eqref{eq:phase_intro} and the flow of the ODE-transport system~\eqref{eq:coupled_DDE} leads to an equation that is not practically tractable. 
In contrast, if we consider a state space where elements are of the form $(c, \psi)$ but drop the condition that $c = \lim_{s \uparrow 0} \psi(s)$, then from a practical perspective it is clear how to conjugate solutions of~\eqref{eq:phase_intro} to solutions of~\eqref{eq:coupled_DDE}--\eqref{eq:coupled_DDE_bc}, but the price that we pay is that we do not establish that every solution of the ODE-transport system~\eqref{eq:coupled_DDE}--\eqref{eq:coupled_DDE_bc} is also a solution of the DDE~\eqref{eq:DDE}. 

The problem outlined above is a specific instance of a problem that is fundamental to the nonlinear theory for DDE, and the solution found in~\cite{diekmann2012} is to mediate between two different state spaces: we first choose a state space (referred to as the `small space') which satisfies requirement (i). Next, we embed this state space in a larger space  which satisfies requirement (ii), and in the last step show that we can restrict to the smaller state space. We expect that conjugacy equations~\eqref{eq:con_full_pt1}--\eqref{eq:con_full_pt2} can be derived using this framework, but rigorously mediating between the small and the big state space requires a substantial amount of duality theory that we do not want to focus on here. Rather, we work under the assumption that solutions are smooth and focus on detailing how the conjugacy equation~\eqref{eq:con_full_pt1}--\eqref{eq:con_full_pt2} leads to a solvable problem in practice. 

\subsection*{Acknowledgement}

We would like to thank Kyle Wedgwood for feedback on previous versions of this manuscript.

\bibliographystyle{siamplain}
\bibliography{reduction}
\end{document}